\documentclass[11pt]{amsart}

\usepackage[pdfusetitle, colorlinks=true]{hyperref}
\usepackage{color}

\usepackage{amsmath, amsthm, amssymb, amsfonts, mathtools}
\usepackage{graphicx,epsfig,color}
\usepackage{exscale}

\newcommand{\red}{{}}

\DeclareMathOperator{\dm}{dm}
\theoremstyle{plain}

\newtheorem{theorem}{Theorem}[section]

\newtheorem{lemma}[theorem]{Lemma}

\newtheorem{corollary}[theorem]{Corollary}

\newtheorem{proposition}[theorem]{Proposition}

\theoremstyle{definition}

\newtheorem{definition}[theorem]{Definition}

\theoremstyle{remark}
\newtheorem{remark}[theorem]{Remark}

\newtheorem{example}[theorem]{Example}

\newtheorem*{jjj}{Remark}

\newtheorem*{remark111}{Remark 1.1.1}

\newtheorem*{141}{Remark 1.4.1}

\DeclareSymbolFont{AMSb}{U}{msb}{m}{n}
\DeclareMathSymbol{\N}{\mathalpha}{AMSb}{"4E}
\DeclareMathSymbol{\R}{\mathalpha}{AMSb}{"52}
\DeclareMathSymbol{\Z}{\mathalpha}{AMSb}{"5A}
\DeclareMathSymbol{\D}{\mathalpha}{AMSb}{"44}
\DeclareMathSymbol{\s}{\mathalpha}{AMSb}{"53}

\DeclareMathOperator{\OptGeo}{OptGeo}

\newcommand{\Y}{{\red Y}}
\DeclareMathOperator{\sgn}{sgn}

\DeclareMathOperator{\lip}{Lip}

\DeclareMathOperator{\supp}{supp}

\DeclareMathOperator{\m}{m}
\DeclareMathOperator{\de}{d}
\DeclareMathOperator{\ric}{ric}
\DeclareMathOperator{\diam}{diam}

\DeclareMathOperator{\Hess}{Hess}

\newcommand{\T}{\mathcal{T}}

\DeclareMathOperator{\sign}{sign}

\newcommand{\CD}{\mathrm{CD}}
\newcommand{\RCD}{\mathrm{RCD}}

\title[Rigidity of mean convex subsets in nonnegatively curved spaces]{Rigidity of  mean convex subsets in non-negatively curved $\RCD$ spaces and stability of  mean curvature bounds}
\author{Christian Ketterer}
\address{University of Freiburg, Mathematical Institute, Ernst-Zermelo-Str. 1, 79104, Germany}
\email{christian.ketterer@math.uni-freiburg.de}
\begin{document}
\maketitle
\begin{abstract} 
We prove splitting theorems for mean convex open subsets in $\RCD$ (Riemannian curvature-dimension) spaces that extend results by Kasue, Croke and Kleiner for Riemannian manifolds with boundary to a non-smooth setting.  A corollary is for instance Frankel's theorem. Then, we prove that our notion of mean curvature bounded from below for the boundary of an open subset is stable w.r.t. to uniform convergence of the corresponding boundary distance function. We apply this to prove  almost rigidity theorems for uniform domains whose boundary has a lower mean curvature bound. 
\end{abstract}
\tableofcontents
\section{Introduction}
By the Cheeger-Gromoll splitting theorem a Riemannian manifold with non-negative Ricci curvature which contains a geodesic line splits off a factor $\R$. In \cite{Kasue83} Kasue proved a version of  this result in the presence of boundary components: A Riemannian manifold with mean convex and compact boundary and nonnegative Ricci curvature that contains a geodesic ray with initial point in the boundary splits off  $[0, \infty)$. Kasue also proved that a  Riemannian manifold with more than one compact mean convex boundary component and non-negative Ricci curvature is isometric to a product $[0,D]\times N$. In particular, there are exactly two boundary components and the mean curvature vanishes. Croke and Kleiner \cite{crokekleiner} showed that this is the special case of a more general splitting principle for Riemannian manifolds with boundary. Generalisations for Bakry-Emery Ricci curvature bounds have been obtained by Sakurai \cite{sakurai} and Moore-Woolgar \cite{mw}.

In this article one of our main  goals is to generalize Kasue's rigidity theorems to the nonsmooth context of $\RCD$ spaces. The latter is the celebrated synthetic notion of Ricci curvature bounded from below for metric measure spaces.  The class of $\RCD$ spaces includes Riemannian manifolds with convex boundary. However Riemannian manifolds that admit boundary with only mean curvature bounded from below are in general not in this class: In the presence of boundary components the interior of a Riemannian manifold may not be geodesically convex and therefore will not satisfy any $\RCD(K,N)$ condition. 
Hence,  for a generalization of Kasue's theorem we consider open subsets inside $\RCD$ spaces whose boundary admits a lower mean curvature bound in a generalized sense. 

In \cite{kettererHK} and in \cite{bkmw}  synthetic notions of lower mean curvature bounds for an open subset $\Omega$ inside an $\RCD$ space $(X,\de,\m)$ were introduced. A similar definition of lower mean curvature bounds in the context of Lorentzian length spaces with synthetic lower Ricci curvature bounds was used in \cite{CavallettiMondino20+}.  Geometric consequences that were derived in \cite{bkmw} are estimates on the inscribed radius of $\Omega$ and rigidity theorems for the corresponding equality cases. 
%In \cite{bkmw} the authors prove a rigidity theorem for this framework that is independent of the number of boundary components and  generalizes another theorem by Kasue that also appears in \cite{Kasue83} (see also \cite{martinli}). 
One of the key steps in the proof of these rigidity theorems is a  comparison estimate for the  Laplacian of the boundary distance function $\de_{\Omega^c}= \inf_{y\in \Omega^c} \de(y, \cdot)$ %of an open set $\Omega\subset X$
 \cite[Corollary 4.11]{bkmw}:
\begin{align}\label{ineq:laplesti}
{\bf \Delta}_{\Omega} (-\de_{\Omega^c})\geq - (N-1) \frac{s'_{\frac{K}{N-1}, \frac{H}{N-1}}(\de_{\Omega^c})}{s_{\frac{K}{N-1}, \frac{H}{N-1}}(\de_{\Omega^c})}\m|_{\Omega}.
\end{align}
Here ${\bf \Delta}_{\Omega}$ is the distributional Laplacian in $\Omega$, $\m|_\Omega$ is the reference measure $\m$ restricted to $\Omega$, $H$ is the synthetic lower mean curvature bound and $$s_{\frac{K}{N-1}, \frac{H}{H-1}}(r)=\cos\left(\scriptstyle{\sqrt{\frac{K}{N-1}}} r\right)- \textstyle{\frac{H}{N-1}}\sin\left(\scriptstyle{\sqrt{\frac{K}{N-1}}} r\right)$$
for $K>0$ and appropriately modified for $K\leq 0$.
In particular, for $K=0$ and $H=\delta(N-1)$ \eqref{ineq:laplesti} becomes $${\bf \Delta}_{\Omega}(-\de_{\Omega^c})\geq {\delta}{(1-\delta\de_{\Omega^c})^{-1}}$$ and by one of the results in \cite{bkmw} one has $\de_{\Omega^c}\leq \frac{1}{\delta}$.
Moreover in \cite{ms21} this Laplace estimate for $H=0$ was  derived for perimeter minimizing sets of finite perimeter in an $\RCD$ space.

In Section \ref{subsec:meancurvature} we will show that under general assumptions on $\partial \Omega$ the Laplace estimate \eqref{ineq:laplesti} is equivalent to the notions of mean curvature bounded from below used in \cite{kettererHK, bkmw}. This is well-known for Riemannian manifolds and justifies the following definition. We will say that the boundary of a general open  subset $\Omega\neq \emptyset$ inside some $\RCD(K,N)$ space $(X,\de,\m)$ has \textit{Laplace mean curvature bounded from below by $H\in \mathbb R$} if the corresponding distance function to the complement $\de_{\Omega^c}$ satisfies \eqref{ineq:laplesti}.
The advantage of this notion for lower mean curvature bounds is that it will work for all open subets $\Omega$ in $\RCD$ spaces without any other a priori assumptions on $\partial \Omega$. Moreover it has nice stability properties.

The first  result of this paper is the following theorem. 
\begin{theorem}\label{main1} 
Let $X$ be an $\RCD(0,N)$ space for $N\geq 1$, and  let $\Omega_\alpha\subset X$, $\alpha=1,\dots, m$ with $m\geq 2$, be open and connected such that $\Omega^c_\alpha\neq \emptyset$ and $\Omega_\alpha^c \cap \Omega_\beta^c= \emptyset$ for $\alpha\neq \beta$. Assume $\partial \Omega_\alpha$  has Laplace mean curvature bounded from below by $0$ for every $\alpha$ and assume that $\partial \Omega_2$ is compact.

Then,  $m=2$ and there exists a metric measure space space $Y$ such that $(\tilde \Omega, \tilde \de_{\Omega}, \m|_{\Omega})$ is isomorphic to $[0, D]\otimes Y$ where $D:=\inf_{x\in \Omega_1^c, y\in \Omega_2^c}\de_X(x, y)$ and $\Omega=\Omega_1\cap \Omega_2$. If $N\geq 2$, then $Y$ is $\RCD(0, N-1)$. If $N\in [1,2)$, then $Y\simeq \{pt\}$.
\end{theorem}
The distance $\tilde \de_\Omega$ is the completion of the  induced intrinsic distance on $\Omega$ and $(\tilde \Omega, \tilde \de_\Omega, \m|_{\Omega})$ is the corresponding metric measure space. 
\begin{remark111} We emphasize that $\tilde \de_\Omega$ cannot be replaced with $\de_X|_\Omega$. A simple counterexample is the $\RCD(0,2)$ space $X$ that is constructed by gluing two copies of a disk $\overline B_1(0)=D$  to the ends of the  cylinder $\mathbb S^1\times [0,1]$. For two points in $\mathbb S^1\times (0,1)=: \Omega$ that are close to $\mathbb S^1 \times \{0\}$ the shortest path w.r.t. $\de_X$ goes through $D$. But $\Omega$ splits w.r.t. the intrinsic distance.
\end{remark111}
%\begin{remark}
%Note that for $N\in [1,2)$, it follows by \cite{kila} that $X$ is isomorphic to $\mathbb S^1$ or to an interval $[0, L]$ equipped with convave measure. Hence, in this case it follows that $(\tilde \Omega, \tilde \de_\Omega, \m|_\Omega)\simeq ([0,D], \mathcal L^1|_{[0,D]})$.
%\end{remark}
As a corollary we obtain 
\begin{corollary}\label{cor1}
Let $X$ be a compact $\RCD(0,N)$ space with $N\geq  2$. There are no open, connected subsets $\Omega_1$ and $\Omega_2$  such that  $\partial \Omega_1$ and $\partial \Omega_2$ are disjoint and have Laplace mean curvature bounded from below by $\delta>0$. 
\end{corollary}
The corollary can be seen as a mean curvature version in context of $\RCD$ spaces of the non-existence result of positive scalar curvature metrics on a torus by Schoen-Yau-Gromov-Lawson \cite{schoenyaustructure, schoenyau79, Gromov-Lawson-1980}.

Another corollary is a Frankel-type theorem for mean convex subsets in positively curved $\RCD$ spaces. 
\begin{corollary}\label{cor2}
Let $X$ be an $\RCD(\delta, N)$ space for $\delta>0$ and $N\geq 2$. Let $\Omega_1$ and $\Omega_2$ be open  connected subsets in $X$ such that $\partial \Omega_1$ and $\partial \Omega_2$ are Laplace mean convex. 
Then $ \Omega^c_1\cap  \Omega_2^c\neq \emptyset$.
\end{corollary}
{The proof that is presented in Section \ref{subsec:isometric} is close to a proof in the Riemannian setting (see \cite{pewifrankel}). A similar result appears in \cite{ms21} for perimeter minimizing sets.}

Putting the boundary  of $\Omega_2$ at infinity in Theorem \ref{main1}, we also get the following theorem.
\begin{theorem}\label{main2}
Let $X$ be an $\RCD(0,N)$ space with $N\geq 1$ and let $\Omega\subset X$ be open and connected with mean curvature bounded from below by $0$. Assume there exists a geodesic ray $\gamma: (0, \infty) \rightarrow \Omega$ with $\lim_{r\downarrow 0} \gamma(r)=x_0\in \partial \Omega\neq \emptyset$ and $\de_X(x_0, \gamma(r))= \de_{\Omega^c}(\gamma(r))$.

Then, there exists a metric measure space $Y$ such that $(\tilde \Omega, \tilde \de_\Omega, \m_{\Omega})$ is isomorphic to $[0,\infty) \otimes Y$. If $N\geq 2$, then $Y$ is $\RCD(0,N-1)$. If $N\in [1,2)$, then $Y\simeq\{pt\}$. 
\end{theorem}
%As before, for $N\in [1,2)$ it follows that $(\tilde \Omega, \tilde\de_\Omega, \m|_\Omega)\simeq ([0,\infty), \mathcal L^1|_{[0,\infty)})$. 
\begin{141}
The assumption $\de_X(x_0, \gamma(r))= \de_{\Omega^c}(\gamma(r))$ for the geodesic ray $\gamma$ cannot be omitted. A counterexample is $X=\R^2$ with $\Omega= \{ (x,y): y=x^2\}$. 
\end{141}
{The proof of Theorem \ref{main1} has two parts.  In Section \ref{subsec:meas} we show that $\Omega$ equipped with the reference measure $\m$ restricted to $\Omega$ splits as measure space.  In Section \ref{subsec:isometric} we  then see that  this implies an isometric splitting for the induced intrinsic geometry of $\Omega$. This part  applies methods developped in \cite{kkl} and we omit details since the steps are  identical with the ones in \cite{kkl}. The proof of Theorem \ref{main2} follows the same roadmap with obvious modifications where we only provide the details of the first part. }

These rigidity results  raise the question for corresponding  almost rigidity theorems: given a Riemannian manifold that satisfies the assumption of the  theorems up to an error $\epsilon$ are we close (and in which sense) to the rigidity case? In absence of extrinsic boundary, that is $\Omega= X$,  these questions can be  answered by $\RCD$ rigidity theorems, stability of $\RCD$ curvature bounds w.r.t. measured Gromov-Hausdorff convergence and Gromov's precompactness theorem. 

For domains with lower mean curvature bounds inside of a Riemannian manifold with Ricci curvature bounded from below the problem is more delicate \cite{peralesheka, wong}.
%The crucial difficulty is the behavior of boundary with a lower mean curvature bound. 
%A  sequence of such spaces, even in combination with lower Ricci curvature bounds in the interior, may not admit a Gromov-Hausdorff converging subsequence in general. For instance bubbling and pinching phenomenas can occur.
%, as well as longer and longer tentacles. 
%Eventually this means that the  
%
A sequence of closed domains may not subconverge in Gromov Hausdorff sense to a metric space. This behavior is  similar to the one of closed Riemannian manifolds with lower scalar curvature bounds (for instance, see \cite{sormanisurvey, gromovmean}). 
%
%One could try to impose stronger regularity assumptions.
%In \cite{wong} Wong showed that in combination with lower and upper bounds on the second fundamental form  of the boundary $\partial \Omega$ one obtains a Gromov-Hausdorff precompactness theorem for $\overline \Omega$ equipped with its induced distance. But even if we are willing to assume this stronger assumptions,  it turns out that this precompactness theorem is not suitable for proving an almost rigidity theorem associated to the Theorem \ref{main1}. A crucial point ist that the rigidity theorems above involve $(\tilde \Omega, \tilde \de_{\Omega})$ and not $\overline \Omega$ that would show up as a limit in Wong's theorem. In fact $\overline \Omega$ and $(\tilde \Omega, \tilde \de_{\Omega})$ can have quite different geometries in general (see the example in Remark \ref{rem:example}).

Our solution to this problem is as follows.
Since we study spaces with boundary as subsets of $\RCD$ spaces, we consider the function $\de_{\Omega^c}$ that is  $1$-Lipschitz. Then we can apply Gromov's Arzela-Ascoli theorem as a  compactness theorem for this framework. For a family of $\RCD(K,N)$ spaces $X_i$ together with functions $\de_{\Omega_i^c}$ one obtains a subsequence of metric measure spaces and distance functions that converge in measured Gromov-Hausdorff sense and uniformily, respectively, to a $1$-Lipschitz function $\de_{\Omega^c}$ on a limit $\RCD$ space $X$. To quantify uniform convergence we introduce the uniform distance between continuous functions (Definition \ref{def:ud}). Applied to  distance functions to the boundary of subsets $\Omega$ and $\Omega'$ in $X$ and in $Y$, respectively, one can define a distance $\mathcal D(\Omega, \Omega')$.
% between such subsets $\Omega$ and $\Omega'$ in compact $\RCD$ spaces $X$ and $Y$. 
Moreover Laplace mean curvature bounds are preserved under this convergence (Theorem \ref{th:meancurvaturestability}).  {The latter is essentially known to experts. For instance, in  \cite{bns} the authors prove a sharp Laplace mean curvature bound for the distance function of the intrinsic boundary of  Ricci limit spaces. }

{These notions yield a compactness statement for pairs $(X,\Omega)$ (Corollary \ref{cor:com})}, and our almost rigidity theorem in the class of subsets in smooth Riemannian manifolds reads as follows.
\begin{theorem}\label{main3} Let $L, c, C, \Gamma \in \R_+$, $N\geq 2$ and $m\in \mathbb N \backslash \{1\}$. For every $\epsilon >0$ there exists $\delta>0$ such that the following holds. 

Let $M$ be a Riemannian manifold with $\ric_M\geq -\delta$, $\dim_M\leq N$ and  $\diam_M\leq L$ and let $\Omega_\alpha\subset X$, $\alpha=1, \dots, m$,  be open subsets with smooth boundary $\partial \Omega_\alpha$ such that $\Omega_\alpha$ is $(c,C)$-uniform, $\partial \Omega_\alpha$ has mean curvature bounded from below by $-\delta$ and $\inf_{x\in \partial \Omega_\alpha, y\in\partial\Omega_\beta}\de_M(x,y)\geq \Gamma>0$ for  $\alpha\neq \beta$. 
%Set $\Omega=\bigcap_{\alpha=1}^m \Omega_\alpha$.

Then, $m=2$ and there exist an $\RCD(0, N)$ space $Z$, an $\RCD(0,N-1)$ space $Y$ and  an open subset $\Omega'\subset Z$ such that
$(\tilde{\Omega}', \tilde \de_{\Omega'}, \m_Z|_{\Omega'}) \simeq Y\otimes [0, D]$ for some $D>0$ and 
\begin{align*}
\mathbb D(X, Z)\leq \epsilon \ \ \ \mbox{ and } \ \ \ \ \mathcal D({\Omega}_1\cap \Omega_2,  {\Omega'})\leq \epsilon.
\end{align*}
\end{theorem}
Here $\mathbb D$ is the Sturm's transportation distance \cite{stugeo1}. We actually will prove the theorem in the class of $\RCD$ spaces.

The main result in \cite{bkmw} is that a subset $\Omega$ with mean curvature bounded from below by $N-1$ inside an $\RCD(0,N)$ space $X$ which attains the inscribed radius bound $1$, is isomorphic to a truncated cone w.r.t. its intrinsic geometry. The following theorem is now the correponding almost rigidity theorem.
\begin{theorem}\label{main4} Let $L, c, C, \Gamma>0$ and $N\geq 2$. For every $\epsilon >0$ there exists $\delta>0$ such that the following holds. 

Let $M$ be a Riemannian manifold with $\dim_M\leq N$, $\ric_M\geq -\delta$ and $\diam_X\leq L$,  and let $\Omega$ be  open and $(c,C)$-uniform such that $\partial \Omega\neq \emptyset$ is smooth and has  mean curvature bounded from below by $N-1-\delta$.  Assume there exists $x\in \Omega$ such that $\de_{\Omega^c}(x)\geq 1-\delta$.

Then, there exists an $\RCD(0, N)$ space $Z$, an $\RCD(N-2,N-1)$ space $Y$ and  an open subset $\Omega'\subset Z$ such that
$(\tilde{\Omega}', \tilde \de_{\Omega'}, \m_Z|_{\Omega'})$ is isomorphic to $Y\times^{N-1}_{r} [0, 1]$ and 
\begin{align*}
\mathbb D(X,Z)\leq \epsilon \ \ \ \mbox{ and } \ \ \ \ \mathcal D(\Omega, \Omega')\leq \epsilon.
\end{align*}
\end{theorem}
Here $Y\times^{N-1}_r [0,1]$ denotes the truncated $N$-Euclidean cone over $Y$.

The notion of $(c,C)$-uniform domain (Definition \ref{def:undo}) is well-known  in the study of elliptic and parabolic PDEs. In our theorem this property guarantees that connectedness of domains is preserved under uniform convergence of their distance functions to the boundary.  
%For instance, it effectively prevents pinching phenomena. 
In fact one can see that connectedness of the limit domain is  necessary to be able to apply the previous rigidity theorems and any assumption on the sequence $\Omega_i$ that preserves connected in the limit will be enough for  the theorem to hold.
%It is however still unsatisfying since it rules out  examples with long and thin spikes or tentacles that disappear in the limit. Therefore, we expect that $(c,C)$-uniformity  can be replaced by a weaker property. We will investigate such a property in a future publication.

Another application of stability of Laplace mean curvature bounds w.r.t. uniform convergence is stability of "constant mean curvature hypersurfaces", and in particular "minimal hypersurfaces", along a sequence of Riemannian manifold with lower Ricci curvature bounds that converge in measured Gromov-Hausdorff sense. We will discuss this in the Appendix \ref{53}.

The article is organized as follows. In Section 2 we recall the necessary background about $\CD$ spaces, first and second order calculus on metric measure spaces, $\RCD$ spaces, and the 1D localisation technique. 

In Section 3 we review several notions of mean curvature bounds for open subsets $\Omega$ in $\RCD$ spaces and show that they are equivalent under suitable regularity assumptions on $\partial \Omega$. In particular, we show  equivalence to the Laplace estimate and introduce Laplace mean curvature bounds. 

In Section 4 we first prove that  open subsets with disconnected boundary and mean curvature bounded from below in essentially non-branching $\CD$ spaces admit a measurable splitting. Then, we obtain the isometric splitting in the context of $\RCD$ spaces (Theorem \ref{main1}, Theorem \ref{main2}).

In Section 5 we first review uniform convergence of functions on a sequence of compact metric spaces, and define the uniform distance. Then we prove stability of mean curvature bounds under uniform convergence and deduce the almost rigidity theorems (Theorem \ref{main3} and Theorem \ref{main4}) in the context of $\RCD$ spaces. 

In the Appendix \ref{53} we prove  Theorem \ref{th:min} concerning "constant mean curvature hypersurfaces". 
\subsection{Acknowledgements}This work  started when the author was funded by the Deutsche Forschungsgemeinschaft (DFG) - Projektnummer 396662902, ``Synthetische Kr\"ummungsschranken durch Methoden des optimal Transports''. The author is grateful to  Daniele Semola for bringing Example \ref{examp} to his attention.  The auhtor also wants to thank the unknown referee for her or his important comments and valuable remarks that helped to improve this article, especially Remark \ref{rem:effective} and the second example in \ref{examp}. \smallskip

\section{Preliminaries.}
\subsection{Curvature-dimension condition}

Let $(X,\de)$ be a complete and separable metric space equipped with a locally finite Borel measure $\m$. We call the triple $(X,\de,\m)$ a metric measure space. Sometimes it will be convenient to simplify the notion and to denote a metric or  metric measure space just $X$ and the correponding distance function or reference measure $\de_X$ and $\m_X$ respectively. We will frequently use this notation in the following. 

Given a metric space $(X,\de)$ a geodesic is a length minimizing curve $\gamma:[a,b]\rightarrow X$.
We denote the set of constant speed geodesics $\gamma:[a,b]\rightarrow X$ with $\mathcal G^{[a,b]}(X)$ equipped with the topology of uniform convergence and set $\mathcal G^{[0,1]}(X)=:\mathcal G(X)$. For $t\in [a,b]$ the evaluation map $e_t:\mathcal G^{[a,b]}(X)\rightarrow X$ is defined as $\gamma\mapsto \gamma(t)$ and $e_t$ is continuous.
%
% More generally, for $[c,d]\subset [a,b]$ we define $e_{[c,d]}: \mathcal G^{[a,b]}(X)\rightarrow \mathcal G^{[c,d]}(X)$ via $e_{[c,d]}(\gamma)=\gamma|_{[c,d]}$ and $e_{[c,d]}$ is also continuous. 
%Given $\gamma\in \mathcal G^{[a,b]}(X)$ we can define $p:\gamma\mapsto p(\gamma)\in \mathcal G(X)$ where $p(\gamma)(t)=\gamma((1-t)a+tb)$. The map $p:\mathcal G^{[a,b]}(X)\rightarrow \mathcal G(X)$ is again continuous w.r.t. uniform convergence.
A set of geodesics $F\subset \mathcal{G}(X)$ is said to be {\it non-branching} if  $\forall\epsilon\in (0,1)$ the map $e_{[0,\epsilon]}|_{F}$ is one to one.

% for any two geodesics $\gamma^1,\gamma^2\in \mathcal G(X)$ the following holds. 
%\begin{align*}
%\mbox{If }\gamma^1|_{[0,\epsilon)}\equiv \gamma^2|_{[0,\epsilon)} \mbox{ for some }\epsilon >0\ \mbox{ then }\ \gamma^1\equiv \gamma^2.
%\end{align*}

The set of (Borel) probability measures is $\mathcal P(X)$, the subset of probability measures with finite second moment is $\mathcal P^2(X)$,  the set of probability measures in $\mathcal P^2(X)$ that are $\m$-absolutely continuous is denoted with $\mathcal P^2(X,\m)$ and the subset of measures in $\mathcal P^2(X,\m)$ with bounded support is denoted with $\mathcal{P}_b^2(X,\m)$.

The space $\mathcal P^2(X)$ is equipped with the $L^2$-Wasserstein distance $W_2$ that is finite on $\mathcal P^2(X)$. 
%
%
%A coupling between $\mu_0,\mu_1\in \mathcal P^2(X)$ is a measure $\pi\in \mathcal P^2(X\times X)$ such that the marginal distributions are $\mu_0$ and $\mu_1$. The set of couplings between $\mu_0$ and $\mu_1$ is $\Cpl(\mu_0,\mu_1)$.  The set of optimal couplings is
%$$
%\OptCpl(\mu_0,\mu_1)=\left\{\pi\in \Cpl(\mu_0,\mu_1):  W_2(\mu_0,\mu_1)^2=\int d(x,y)^2 d\pi(x,y)\right\}.
%$$
A dynamical optimal coupling is a probability measure $\Pi\in \mathcal P(\mathcal G(X))$  such that $t\in [0,1]\mapsto (e_t)_{\#}\Pi$ is a $W_2$-geodesic in $\mathcal P^2(X)$. 
The set of  dynamical optimal couplings $\Pi\in \mathcal P(\mathcal G^{}(X))$ between $\mu_0,\mu_1\in \mathcal P^2(X)$ is denoted with $\OptGeo(\mu_0,\mu_1)$. 
%We set again $\OptGeo^{[0,1]}(X)=:\OptGeo(X)$. 
%We say $\Pi\in \OptGeo^{[a,b]}(\mu_a,\mu_b)$ is induced by a map $T:X\rightarrow \mathcal G^{[a,b]}(X)$ if $T_{\#}\mu_a=\Pi$. Given $\Pi\in \OptGeo^{[a,b]}(\mu_a,\mu_b)$ we can consider $(e_{[c,d]})_{\#}\Pi:=\Pi'\in \mathcal G^{[c,d]}(X)$ for $[c,d]\subset [a,b]$. It follows that $r\in [c,d]\mapsto (e_r)_{\#}\Pi'=(e_r)_{\#}\Pi$ is an $W_2$-geodesic 
%between $\mu_c$ and $\mu_d$ and hence $\Pi'\in \OptGeo^{[c,d]}(\mu_c,\mu_d)$. Given $\Pi\in \OptGeo^{[a,b]}(\mu_a,\mu_b)$ we always can consider the pushforward $p_{\#}\Pi\in \mathcal P(\mathcal G(X))$. Then, $t\in [0,1]\mapsto (e_t)_{\#} p_{\#}\Pi= (e_{(1-t)a+tb})_{\#}\Pi$ is an $W_2$-geodesic, and hence $p_{\#}\Pi\in \OptGeo(\mu_a,\mu_b)$.
%
%A probability measure $\Pi\in \mathcal P(\mathcal G(X))$ is concentrated on a set of nonbranching geodesics $F$ if $\Pi(F)=1$. 

A metric measure space $(X,\de,\m)$ is called \textit{essentially nonbranching} if for any pair $\mu_0,\mu_1\in \mathcal P^2(X,\m)$ every optimal dynamical plan $\Pi\in \OptGeo(\mu_0,\mu_1)$ is concentrated on a set of nonbranching geodesics.
\begin{definition}\label{def:gensin}
For $\kappa\in \mathbb{R}$ we define $\cos_{\kappa}:[0,\infty)\rightarrow \mathbb{R}$ as the solution of 
\begin{align*}
v''+\kappa v=0, \ \ \ v(0)=1 \ \ \& \ \ v'(0)=0.
\end{align*}
$\sin_{\kappa}$ is defined as solution of the same ODE with initial value $v(0)=0 \ \&\ v'(0)=1$. That is 
\begin{align*}
\cos_{\kappa}(x)=\begin{cases}
 \cosh (\sqrt{|\kappa|}x) & \mbox{if } \kappa<0\\
1& \mbox{if } \kappa=0\\
\cos (\sqrt{\kappa}x) & \mbox{if } \kappa>0
                \end{cases}
                \quad
   \sin_{\kappa}(x)=\begin{cases}
\frac{ \sinh (\sqrt{|\kappa|}x)}{\sqrt{|\kappa|}} & \mbox{if } \kappa<0\\
x& \mbox{if } \kappa=0\\
\frac{\sin (\sqrt{\kappa}x)}{\sqrt \kappa} & \mbox{if } \kappa>0
                \end{cases}                 
                \end{align*}
Let $\pi_\kappa$ be the diameter of a simply connected space form $\mathbb S^2_k$ of constant curvature $\kappa$, i.e.
\[
\pi_\kappa= \begin{cases}
 \infty \ &\textrm{ if } \kappa\le 0\\
\frac{\pi}{\sqrt \kappa}\ &  \textrm{ if } \kappa> 0.

\end{cases}
\]
In \cite{Kasue83} and \cite{sakurai} the authors define 
\begin{align}
\label{equ:ks}
s_{\kappa,\lambda}{\red (r)} = \cos_{\kappa}{\red (r)}- \lambda \sin_\kappa{\red (r)}
\end{align}
for $\kappa,\lambda\in \R$. The pair $(\kappa,\lambda)$ satisfies {\it  the ball condition} if the equation $s_{\kappa, \lambda}(r)=0$ has a positive solution. The latter happens if and only if one of the following three cases holds:  (1) $\kappa>0$ and $\lambda\in \R$, (2) $\kappa=0$ and $\lambda>0$ or (3) $\kappa\leq 0$ and $\lambda>\sqrt{|\kappa|}$.
{For $(\kappa,\lambda)=(\frac{K}{N-1},\frac{H}{N-1})$  let $r_{K,H,N}$ be  the smallest positive zero of $s_{\kappa,\lambda}= s_{K/(N-1), H/(N-1)}$ if any exists;  
moreover $s_{\kappa,\lambda}(r)<0$ for all $r>r_{K,H,N}$ if $\kappa \le 0$,   while $s_{\kappa,\lambda}$ oscillates sinusoidally with mean zero and period 
greater than $2 r_{K,H,N}$ if $\kappa>0$. In particular,}  
$r_{K,H,N}<\infty$ if and only if $(\frac{K}{N-1},\frac{H}{N-1})$ satisfies the ball-condition.

For $K\in \mathbb{R}$, $N\in (0,\infty)$ and $\theta\geq 0$ we define the \textit{distortion coefficient} as
\begin{align*}
t\in [0,1]\mapsto \sigma_{K,N}^{(t)}(\theta)=\begin{cases}
                                             \frac{\sin_{K/N}(t\theta)}{\sin_{K/N}(\theta)}\ &\mbox{ if } \theta\in [0,\pi_{K/N}),\\
                                             \infty\ & \ \mbox{otherwise}.
                                             \end{cases}
\end{align*}
Note that $\sigma_{K,N}^{(t)}(0)=t$.
Moreover, for $K\in \mathbb{R}$, $N\in [1,\infty)$ and $\theta\geq 0$ the \textit{modified distortion coefficient} is defined as
\begin{align*}
t\in [0,1]\mapsto \tau_{K,N}^{(t)}(\theta)=\begin{cases}
                                            \theta\cdot\infty \ & \mbox{ if }K>0\mbox{ and }N=1,\\
                                            t^{\frac{1}{N}}\left[\sigma_{K,N-1}^{(t)}(\theta)\right]^{1-\frac{1}{N}}\ & \mbox{ otherwise}
                                           \end{cases}\end{align*}
where our convention is $0\cdot \infty =0$.
\end{definition}
%
%\begin{lemma}[\cite{stugeo2}]\label{lem:coef}
%For all $K, K'\in \R$, all $N, N'\in (0,\infty)$, all $t\in [0,1]$ and all $\theta\in (0,\infty)$, it holds
%\begin{align*}
%\sigma_{K,N}^{(t)}(\theta)^N \sigma_{K',N'}^{(t)}(\theta)^{N'}\geq \sigma_{K+K',N+N'}^{(t)}(\theta)^{N+N'},
%\end{align*}
%and for $N\geq 1$
%\begin{align*}
%\tau_{K,N}^{(t)}(\theta)^N \sigma_{K',N'}^{(t)}(\theta)^{N'}\geq \tau_{K+K',N+N'}^{(t)}(\theta)^{N+N'}.
%\end{align*}
%\end{lemma}
\begin{definition}[Curvature-Dimension Condition, \cite{stugeo2,lottvillani, bast}]
A metric measure space $(X,\de,\m)$ satisfies the \textit{curvature-dimension condition} $\CD(K,N)$, $K\in \mathbb{R}$, $N\in [1,\infty)$, if for all $\mu_0,\mu_1\in \mathcal{P}_b^2(X,\m)$ 
there exists an $L^2$-Wasserstein geodesic $(\mu_t)_{t\in [0,1]}$ and an optimal coupling $\pi$ between $\mu_0$ and $\mu_1$ such that 
\begin{align}\label{ineq:cd}
S_N(\mu_t|\m)\leq -\int \left[\tau_{K,N}^{(1-t)}(\theta)\rho_0(x)^{-\frac{1}{N}}+\tau_{K,N}^{(t)}(\theta)\rho_1(y)^{-\frac{1}{N}}\right]d\pi(x,y)
\end{align}
where $\mu_i=\rho_id\m$, $i=0,1$, and $\theta= \de(x,y)$.

{We say a metric measure space $(X,\de,\m)$ satisfies the \textit{reduced curvature-dimension condition} $\CD^*(K,N)$ for $K\in \mathbb{R}$ and $N\in (0,\infty)$ if we replace the coefficients $\tau^{(t)}_{K,N}(\theta)$ with $\sigma_{K,N}^{(t)}(\theta)$.}

If $(X,\de,\m)$ is complete and satisfies the condition $\CD(K,N)$ for $N<\infty$, then $(\supp \m, \de)$ is a geodesic space and $(\supp\m,  \de,\m)$ is 
$\CD(K,N)$. In the following we can always assume that $\supp\m=X$. The condition $\CD(K,N)$ implies the condition $\CD^*(K,N)$. 
\end{definition}

\subsection{First order calculus for metric measure spaces}
%
%{\color{blue} For further details about the { properties of $\RCD$ spaces} we refer to
%\cite{agslipschitz,agsheat,agsriemannian,giglistructure, gmsstability}.}

Let $(X,\de,\m)$ be a metric measure space. We denote with $\lip(X)$ the space of Lipschitz functions $f:X\rightarrow \R$, and with $\lip_c(\Omega)$ the space of Lipschitz function with support in $\Omega$ for an open set $\Omega\subset X$.
For $f\in \lip(X)$ the local slope is
\begin{align*}
\mbox{Lip}(f)(x)=\limsup_{y\rightarrow x}\frac{|f(x)-f(y)|}{d(x,y)}, \ \ x\in X.
\end{align*}
If $f\in L^2(\m)$, a function $g\in L^2(\m)$ is called \textit{relaxed gradient} if there exists a sequence of Lipschitz functions $f_n$ which converges in $L^2(\m)$ to $f$, and there exists $h$ such that 
$\mbox{Lip}f_n$ weakly converges to $h$ in $L^2(\m)$ and $h\leq g$ $\m$-a.e. A function $g\in L^2(\m)$ is called the \textit{minimal relaxed gradient} of $f$ and denoted by $|\nabla f|$ if it is a relaxed gradient and minimal w.r.t. the $L^2$-norm among all relaxed gradients.
The object $|\nabla f|$ is local in  the sense that 
\begin{align*}
|\nabla f|= 0 \ \m\mbox{-a.e. on } f^{-1}(\mathcal N) \ \forall \mathcal N\subset \R \mbox{ s.t. } \mathcal L^1(\mathcal N)
\end{align*}
and 
$
|\nabla f|=|\nabla g| \ \m\mbox{-a.e. on } \{f=g\}.$
The space of \textit{$L^2$-Sobolev functions} is $$W^{1,2}(X):= \left\{ f\in L^2(\m): \int |\nabla f|^2 d\m<\infty\right\}.$$
The set $W^{1,2}(X)$ equipped with the norm 
$
\left\|f\right\|_{W^{1,2}(X)}^2=\left\|f\right\|^2_{L^2}+\left\||\nabla f|\right\|_{L^2}^2
$
is a Banach space.
If $W^{1,2}(X)$ is a Hilbert space, we say the metric measure space $(X,\de,\m)$ is \textit{infinitesimally Hilbertian.}

For $f,g\in W^{1,2}(X)$ one defines functions $D^{\pm}f(\nabla g): X\rightarrow \R$ by
\begin{align*}
D^+f(\nabla g) = \inf_{\epsilon>0} \frac{ |\nabla (f+\epsilon g)|^2 - |\nabla f|^2}{2\epsilon}, \\ 
D^-f(\nabla g) = \sup_{\epsilon<0} \frac{ |\nabla (f+\epsilon g)|^2 - |\nabla f|^2}{2\epsilon}.
\end{align*}
If $(X,\de,\m)$ is infinitesimally Hilbertian, then  $D^+f(\nabla g)= D^-f(\nabla g)$ $\m$-a.e.
Moreover
\begin{align}\label{rcdinnerproduct}
 \langle \nabla f,\nabla g\rangle := D^+f(\nabla g)= D^+g(\nabla f)=\frac{1}{4}|\nabla (f+g)|^2-\frac{1}{4}|\nabla (f-g)|^2
\end{align}
and $\langle \nabla f, \nabla g\rangle\in L^1(\m)$.

\subsection{Distributional Laplacian and strong maximum principle}We recall the notion of {the} distributional Laplacian (cf. \cite{giglistructure, cav-mon-lapl-18}).

\begin{definition} Let $(X,\de,\m)$ be a locally compact metric measure  {space and $\Omega\subset X$ be open. Let $\lip_c(\Omega)$ denote the set of Lipschitz functions compactly supported in an open subset  $\Omega$.}
A {\em Radon functional} over $\Omega$  is a linear map $T: \lip_c(\Omega)\rightarrow \R$ such that for every compact subset $W$ in $\Omega$ there exists a constant $C_W\geq  0$ such that
\begin{align}\label{Radon}
|T(f)|\leq {C_W} \max_W|f| \qquad \forall f\in \lip_c(\Omega) \mbox{ with }\supp f\subset W.
\end{align}
One says $T$ is non-negative if {\red $T(f)\geq 0$ for all $f\in \lip_c(\Omega)$ satisfying $f\geq 0$.}
\end{definition}
\begin{remark}
The Riesz-Markov-Kakutani representation theorem says that for a non-negative Radon functional {$T$}
 there exists a  unique Radon measure $\mu_T$ such that $T(f)=\int f d\mu_T$ $\forall$ $f\in \lip_c(\Omega)$.
 \end{remark}
Recall that $u\in W_{loc}^{1,2}(\Omega)$ for an open set $\Omega\subset X$ if for any Lipschitz function $\phi$ with compact support in $\Omega$ we have $\phi \cdot u\in W^{1,2}(X)$. Thanks to the locality properties of $|\nabla f|$ for $f\in W^{1,2}(X)$ the object $|\nabla u|$ is well defined $\m$-a.e. for $u\in W^{1,2}_{loc}(\Omega)$. If $|\nabla u|\in L^2(\m)$, one writes $u\in W^{1,2}(\Omega)$. 
If $u\in \lip(X)$ then $u\in W_{loc}^{1,2}(\Omega)$. 

\begin{definition}[{Nonsmooth Laplacian}]
Let $\Omega\subset X$ be  open and  let $u\in W^{1,2}(\Omega)$. One says $u$ is in the domain of the distributional Laplacian on $\Omega$, writing $u\in D({\bf \Delta}, \Omega)$, provided there exists a Radon functional $T$ over $\Omega$ such that
\begin{align*}
\int D^-u(\nabla f) d\m \leq - T(f) \leq \int  D^+ u(\nabla f) d\m {\quad \forall f \in \lip_c(\Omega)}.
\end{align*}If $T$ is represented as a measure $\mu_T$, one writes $\mu_T\in {\bf \Delta}_{\Omega}u$. If there is only one such measure $\mu_T$ by abuse of notation we will identify $\mu_T$ with $T$ and write $\mu_T={\bf \Delta}_{\Omega} u$.

If $(X,\de,\m)$ is infinitesimally Hilbertian,   $u\in W^{1,2}(X)$ is in the domain of the $L^2$-Laplacian if there exists $h\in L^2(\m)$ such that 
\begin{align*}
\int \langle \nabla u, \nabla f\rangle d\m = \int hf d\m \ \ \forall f\in \lip(X).
\end{align*}
In this case we write $h=\Delta u$ and $u\in D_{L^2}(\Delta)$. For a linear subspace $\mathbb V\subset L^2(\m)$ we write $u\in D_{\mathbb V}(\Delta)$ whenever $\Delta u\in \mathbb V$. 
\end{definition}
\begin{remark}[{Locality and linearity}]
\begin{itemize}
\item[(i)] If $u\in D({\bf \Delta}, \Omega)$ and $\Omega'$ is open in $X$ with $\Omega'\subset \Omega$, then $u\in D({\bf \Delta},\Omega')$ and for $\mu\in {\bf \Delta}_{\Omega}u$ it follows that $\mu|_{\Omega'}\in {\bf \Delta}_{\Omega'}u$.
\item[(ii)] Assume $(X,\de,\m)$ is infinitesimally Hilbertian.  If $u,v\in D({\bf \Delta}, \Omega)$, then $u+v\in D({\bf \Delta},\Omega)$ and for $\mu_u\in {\bf \Delta}_\Omega u$ and $\mu_v\in {\bf \Delta}_{\Omega}v$ it follows that $\mu_u+\mu_v\in {\bf \Delta}_\Omega(u+v)$.
\end{itemize}
\end{remark}

Recall that $u\in W^{1,2}(\Omega)$ is {\it sub-harmonic} if
\begin{align*}
\int_{\Omega} |\nabla u|^2 d\m\leq \int_{\Omega} |\nabla (u+g)|^2 d\m \ \ \forall g\in W^{1,2}(\Omega) \mbox{ with } g\leq 0.
\end{align*}
One says $u$ is {\it super-harmonic} if $-u$ is sub-harmonic, and $u$ is {\it harmonic} if it is both sub- and super-harmonic.  

%\marginpar{\blue Shouldn't we specify $(K,N)$ in Theorem \ref{thm:sh} as in Theorem \ref{thm:mp}?}

\begin{theorem}[{Characterizing super-harmonicity, \cite[Theorem 4.3]{giglimondino}}]\label{thm:sh}
Let $X$ be {\red an $\RCD(K,N)$ space with $K\in \R$ and $N\in [1,\infty)$,} let $\Omega \subset X$ be open and $u\in W_{loc}^{1,2}(\Omega)$.
Then $u$ is {super-harmonic} if and only if ${u} \in D({\bf \Delta}, \Omega)$ and there exists $\mu\in {\bf \Delta}_\Omega u$ such that $\mu \leq 0$.
\end{theorem}

The following is \cite[Theorem 9.13]{bjoern} (see also \cite{gigli_rigoni}):
\begin{theorem}[Strong Maximum Principle]\label{thm:mp}
Let $X$ be an $\RCD(K,N)$ space with $K\in \R$ and $N\in [1,\infty)$, let $U\subset X$ be {\red a connected open set} with compact closure and let $u\in W_{loc}^{1,2}(\Omega)\cap C(\Omega)$ be sub-harmonic.
If there exists $x_0\in \Omega$ such that $u(x_0)=\max_{\bar \Omega} u$ then $u$ is constant.
\end{theorem}
%
%
%{\color{blue}
%Assuming $X$  is locally compact, if $\Omega$ is an open subset of $X$, we say that $f\in W^{1,2}(X)$ is in the domain $D({\bf \Delta},\Omega)$ of the \textit{measure-valued Laplace} ${\bf \Delta}$ on $\Omega$ if there exists a signed Radon functional ${\bf \Delta}f$ on the set of all Lipschitz functions $g$ with bounded support in $\Omega$ such that
%\begin{align}\label{equ:integrationbyparts}
%\int\langle \nabla g,\nabla f\rangle dm = -\int g d{\bf \Delta}f.
%\end{align}
%If $\Omega=X$ and ${\bf \Delta}f= [{\bf\Delta} f]_{ac} \m$ with $[{\bf\Delta} f]_{ac}\in L^2(\m)$, we write $[{\bf\Delta} f]_{ac}=:\Delta f$ and $D({\bf \Delta}, X)=D_{L^2(\m)}(\Delta)$. 
%$\mu_{ac}$ denotes the $\m$-absolutely continuous part in the Lebesgue decomposition of a Borel measure $\mu$.}

\subsection{Riemannian curvature-dimension conditon} 
\begin{definition}\label{def:rcd}
A metric measure space $(X,\de,\m)$ satisfies the (reduced) \emph{Riemannian curvature-dimension condition} $\RCD(K,N)$ ($\RCD^*(K,N)$) for $K\in \mathbb{R}$ and {$N\in [1,\infty)$} if it satisfies the (reduced) curvature-dimension
condition $\CD(K,N)$ ($\CD^*(K,N)$) and is infinitesimally Hilbertian. 
%Similarly, one defines the condition $\RCD^*(K,N)$. 
\end{definition}
For a brief overview on the history of this definition we refer  the reader to the preliminary section of \cite{Kap-Ket-18}.  {For $N\in [1,\infty)$ an $\RCD^*(K,N)$ space  $X$ with $\m_X$ finite satisfies the condition $\RCD(K,N)$ \cite{cavmil} and  the converse direction holds without any assumption. }

%The following was introduced in \cite{agsbakryemery}.
Let $(X,\de,\m)$ be a metric measure space that is infinitesimally Hilbertian but does not necessarily satisfy a curvature-dimension condition. For $f\in D_{W^{1,2}(X)}(\Delta)$ and $\phi\in D_{L^{\infty}}(\Delta)\cap L^{\infty}(\m)$ the \textit{carr\'e du champ operator} is defined as
\begin{align*}
\Gamma_2(f;\phi):=\int \frac{1}{2}|\nabla f|^2\Delta \phi d\m - \int\langle\nabla f,\nabla \Delta f\rangle \phi d\m.
\end{align*}
A metric measure space $(X,\de,\m)$ satisfies the \textit{Bakry-\'Emery condition} $BE(K,N)$ for $K\in \mathbb{R}$, $N\in (0,\infty]$ if it satisfies the weak Bochner inequality
\begin{align*}
\Gamma_2(f;\phi)\geq \frac{1}{N}\int (\Delta f)^2 \phi d\m + K\int |\nabla f|^2 \phi d\m.
\end{align*}
for any $f\in D_{W^{1,2}(X)}(\Delta)$ and $\phi\in D_{L^{\infty}}(\Delta)\cap L^{\infty}(\m), \, \phi\ge 0$.

A metric measure space satisfies the \textit{Sobolev-to-Lipschitz} property if every
$f\in W^{1,2}(X)$ with $|\nabla f|\in L^\infty(\m)$ admits a Lipschitz representative $\tilde f\in \lip(X)$ such that the local Lipschitz constant is bounded from above $\left\| |\nabla f|\right\|_{L^\infty}$.
For $\RCD$ spaces the Sobolev-to-Lipschitz property was proved in \cite[Theorem 6.2]{agsriemannian}.
\begin{theorem}[\cite{erbarkuwadasturm, agsbakryemery, amsnonlinear}]\label{th:be}
Let $(X,\de,\m)$ be a metric measure space. The reduced Riemannian curvature-dimension condition \linebreak[4] $\RCD^*(K,N)$ for $K\in \mathbb{R}$ and $N\in [1,\infty]$ holds if and only if $(X,\de,\m)$ is infinitesimally Hilbertian, satisfies the 
Sobolev-to-Lipschitz property and the exponential growth condition $\int e^{-C\de(x_0, \cdot)^2} d\m$ for some $x_0\in X$,  and satisfies the Bakry-Emery condition $BE(K,N)$.
\end{theorem}

An important class of functions on an $\RCD$ space $(X,\de,\m)$ is the family $\mathbb D_\infty$ of test functions that is defined by
\begin{align*}
\mathbb D_\infty= \big\{ f\in D_{W^{1,2}(X)} (\Delta)\cap L^\infty(\m): |\nabla f|\in L^{\infty}(\m)\big\}.
\end{align*}
For $f\in \mathbb D_\infty$ one can define a Hessian $\mbox{Hess}(f)$ via the formula 
\begin{align*}
&2\mbox{Hess}f(\nabla g,\nabla h)=\\
&  \langle \nabla g, \nabla \langle \nabla h, \nabla f\rangle \rangle + \langle \nabla h, \nabla \langle \nabla f, \nabla g\rangle \rangle - \langle \nabla f, \nabla \langle \nabla g, \nabla h\rangle \rangle\mbox{ for $g, h\in \mathbb D_\infty$.}
\end{align*}
One can extend the operator $\mbox{Hess}$ to the bigger class $H^{2,2}(X)$ that contains $\mathbb{D}_\infty$ and $D_{L^2}(\Delta)$. For $f\in H^{2,2}(X)$ the Hessian is a tensorial object and admits a Hilbert-Schmidt norm $|\mbox{Hess} f|_{HS}\in L^2(\m)$. 
%For $f\in W^{2,2}(X)$ is tensor object. 
%\begin{theorem}[\cite{giglistructure}, \cite{gigtam}]\label{th:secondvariation}
%Let $(X,d,\m)$ be a metric measure space that satisfies the condition $\RCD(K,N)$ for $N<\infty$. Let $\mu_0,\mu_1\in \mathcal{P}^1(X)$ such $\mu_i=\rho_i\m\leq C\m$ for $C>0$ and $i=0,1$, and let $(\mu_t)_{t\in [0,1]}$ be 
%the unique $L^2$-Wasserstein geodesic. 
%\begin{itemize}
%\item[(i)] {\textit First variation formula.} Let $f\in W^{1,2}(X)$. Then, the map $t\in [0,1]\rightarrow \int f d\mu_t$ belongs to $C^1([0,1])$ and for every $t\in [0,1]$ it holds
%\begin{align*}
%\frac{d}{dt}\int f d\mu_t={\int \langle \nabla f,\nabla \phi_t\rangle d\mu_t.}
%\end{align*}
%\item[(ii)] {\textit Second variation formula.}
%Moreover, let $f\in H^{2,2}(X)$. Then, the map $t\in [0,1]\rightarrow \int fd\mu_t$ belongs to $C^2([0,1])$ and for every $t\in [0,1]$ it holds
%\begin{align*}
%\frac{d^2}{dt^2}\int fd\mu_t=\int \Hess f(\nabla\phi_t,\nabla \phi_t)d\mu_t
%\end{align*}
%where $\phi_t$ is the function such that for some $t\neq s\in [0,1]$ the function $-(s-t)\phi_t$ is a Kantorovich potential between $\mu_t$ and $\mu_s$.
%\end{itemize}
%\end{theorem}
\begin{theorem}[\cite{savareself, giglinonsmooth, sturmconformal}]If the metric measure space $(X,\de,\m)$ satisfies the  Riemannian curvature-dimension condition $\RCD(K,\infty)$, and $f\in \mathbb{D}_{\infty}$, then $|\nabla f|^2\in W^{1,2}(X)\cap D({\bf \Delta})$ and
an improved Bochner formula holds in the sense of measures involving the Hilbert-Schmidt norm of the Hessian of $f$:
\begin{align*}
{\bf \Gamma}_2(f):=\frac{1}{2}{\bf \Delta}|\nabla f|^2- \langle\nabla f,\nabla \Delta f\rangle \m \geq \left[K |\nabla f|^2 + |\Hess f|_{HS}^2 \right]\m
\end{align*}
where ${\bf\Delta}|\nabla f|^2$ is given by unique measure, and ${\bf \Gamma}_2$ is called {\it measure valued $\Gamma_2$-operator}. In particular, the singular part of the left hand side in previous inequality is non-negative.
\end{theorem}

\subsection{$1D$-localization}\label{subsec:1D}

\noindent
In this section we will recall  basic facts about the localization technique introduced by Cavalletti and Mondino {for 1-Lipschitz
functions as a nonsmooth analogue of Klartag's needle decomposition: needle refers to any geodesic along
which the Lipschitz function attains its maximum slope, {\red also called transport rays here and by Klartag and others \cite{EvansGangbo99,FeldmanMcCann02,Klartag17}.}}
 The presentation follows Sections 3 and 4 in \cite{cavmon}. We assume familiarity with basic concepts in optimal transport (for instance \cite{viltot}).

Let $(X,\de,\m)$ be a proper metric measure space with $\supp\m =X$ as we always assume.

Let $u:X\rightarrow \mathbb{R}$ be a $1$-Lipschitz function. Then  {\red the \it transport ordering}
\begin{align*}
\Gamma_u:=\{(x,y)\in X\times X : u(y)-u(x)=\de(x,y)\}
\end{align*}
is a  $\de$-cyclically monotone set, and one defines %{\red the \it transport ordering} 
$\Gamma_u^{-1}=\{(x,y)\in X\times X: (y,x)\in \Gamma_u\}$.

Note that we switch orientation in comparison to \cite{cavmon} where Cavalletti and Mondino define $\Gamma_u$ as $\Gamma_u^{-1}$.

The union $\Gamma_u \cup \Gamma_u^{-1}$ defines a relation $R_u$ on $X\times X$, and $R_u$ induces the {\it transport set with endpoints and branching points}
$$\mathcal T_{u,e}:= P_1(R_u\backslash \{(x,y):x=y\in X\})\subset X$$
where $P_1(x,y)=x$. For $x\in \T_{u,e}$ one defines $\Gamma_u(x):=\{y\in X:(x,y)\in \Gamma_u\}, $
and similarly  $\Gamma_u^{-1}(x)$ and $R_u(x)$. Since $u$ is $1$-Lipschitz,
 $\Gamma_u, \Gamma_u^{-1}$ and $R_u$ are closed{\red ,  as are} $\Gamma_u(x), \Gamma_u^{-1}(x)$ and $R_u(x)$.

The sets of {\it forward} and {\it backward branching points}, $A_+ \ \&\ A_-$, are defined respectively as
\begin{align*}
A_{+/-}\!:=\!\{x\in \mathcal T_{u,e}: \exists z,w\in \Gamma_u(x)/\Gamma_u^{-1}(x) \mbox{ \& } (z,w)\notin R_u\}.
%\
%A_{-}\!:=\!\{x\in \mathcal T_{u,e}: \exists z,w\in  \Gamma_u^{-1}(x) \mbox{ \& } (z,w)\notin R_u\}.
%
\end{align*}
Then one considers the {\it (nonbranched) transport set} as $\mathcal T_u:=\mathcal T_{u,e}\backslash (A_+ \cup A_-)$ and the {\it (nonbranched) transport relation} as the restriction of $R_u$ to $\mathcal T_u\times \mathcal T_u$.

The sets
$\T_{u,e}$, $A_{+}$ and $A_-$ are $\sigma$-compact (\cite[Remark 3.3]{cavmon} and \cite[Lemma 4.3]{cavom} respectively), and $\T_u$ is a  Borel set.
In \cite[Theorem 4.6]{cavom} Cavalletti shows that the restriction of $R_u$ to $\mathcal T_u\times \mathcal T_u$ is an equivalence relation.
Hence, from $R_u$ one obtains a partition of $\mathcal T_u$ into a disjoint family of equivalence classes $\{X_{\alpha}\}_{\alpha\in Q}$. A section is a map $s:\T_u\rightarrow \T_u$  such that if $(x,s(x)) \in R_u$ and $(y,x)\in R_u$ then $s(x)=s(y)$. By  \cite[Proposition 5.2]{cavom} there exists a measurable section $s$, and the quotient space $Q$ can be identified with the image of $\T_u$ under this map $s$.  Hence,  we  can and will  consider $Q$ as a subset of $X$, namely the image of $s$, equipped with the induced measurable structure

The quotient map $\mathfrak Q:\T_u\rightarrow Q$  given by the measurable section $s$ is measurable, and we set $\mathfrak q:= \mathfrak Q_{\#}\left[\m|_{\T_u}\right]$. {Hence $\mathfrak q$ is a Borel measure on $X$. By inner regularity we replace $Q$ with a Borel set $Q'\subset Q$ such that $\mathfrak q(Q\backslash Q')=0$ and in the following we denote $Q'$ by $Q$}  {(compare with \cite[Proposition 3.5]{cavmon} and the following remarks).}

Every $X_{\alpha}$, $\alpha\in Q$, is isometric to an interval
$I_\alpha\subset\mathbb{R}$ (c.f. \cite[Lemma 3.1]{cavmon} and the comment after Proposition 3.7 in \cite{cavmon}) via a distance preserving map $\gamma_{\alpha}:I_\alpha \rightarrow X_{\alpha}$ where $\gamma_\alpha$ is parametrized such that $d(\gamma_\alpha(t), s(\gamma_\alpha(t)))=\sgn(\gamma_\alpha(t))t$, $t\in I_\alpha$, and  where $\sgn x$ is the sign of $u(x)-u(s(x))$. The map $\gamma_{\alpha}:I_\alpha\rightarrow X$ extends to a geodesic also  denoted $\gamma_{\alpha}$ and defined on the closure $\overline{I}_{\alpha}$ of $I_{\alpha}$. We set $\overline{I}_{\alpha}=[a(X_{\alpha}),b(X_{\alpha})]$.

%\marginpar{Above consistency was only defined for $\m(X)<\infty$;  but now we allow $\m(X)=\infty$?}

In \cite[Theorem 3.3]{cav-mon-lapl-18}, Cavalletti and Mondino 
%
%extend Definition \ref{D:disintegration} to disintegrate measures $\m$ which are merely $\sigma$-finite
%by using a positive function $f$ on $X$ to relate $\m$ to a probability measure $f(x)d\m(x)$.  Using the framework of this extension, which we also adopt, they 
prove:

\begin{theorem}[Disintegration into needles/transport rays]\label{T:CM disintegration}
Let $(X,\de,\m)$ be a  geodesic metric measure space with $\supp\m =X$ and $\m$ $\sigma$-finite. Let $u:X\rightarrow \mathbb{R}$ be a $1$-Lipschitz function,  let $\{X_{\alpha}\}_{\alpha\in Q}$ be the induced partition of $\mathcal T_u$ via $R_u$, and let $\mathfrak Q: \T_u\rightarrow Q$ be the induced quotient map as above.
Then, there exists a unique {strongly consistent disintegration} $\{\m_{\alpha}\}_{\alpha\in Q}$ of $\m|_{\T_u}$ {with respect to} $\mathfrak Q$.
\end{theorem}
The following is {\cite[Lemma 3.4]{cav-mon-lapl-18}}.
\begin{lemma}[Negligibility of branching points]
\label{somelemma}
Let $(X,d,\m)$ be an essentially nonbranching $MCP(K,N)$ space, $K\in \R$, $N\in (1,\infty)$, with $\supp \m=X$.
Then, for any $1$-Lipschitz function $u:X\rightarrow \R$, it follows $\m(\T_{u,e}\backslash \T_u)=0$.
\end{lemma}
The initial and final points are defined  by
\begin{align*}
\mathfrak a_u \!:=\! \left\{ x\in \T_{u,e}: \Gamma^{-1}_u(x)=\{x\}\right\}, \ \
\mathfrak b_u \!:=\! \left\{ x\in \T_{u,e}: \Gamma_u(x)=\{x\}\right\}.
\end{align*}
In \cite[Theorem 7.10]{cavmil} it was proved that under the assumption of the previous lemma there exists $\hat Q\subset Q$ with $\mathfrak q(Q\backslash \hat Q)=0$ such that for $\alpha\in \hat Q$ one has $\overline{X_\alpha}\backslash \T_u\subset \mathfrak a_u\cup \mathfrak b_u$.
In particular, for $\alpha\in \hat Q$ we have
\begin{align}\label{somehow}
R_u(x)=\overline{X_\alpha}\supset X_\alpha \supset (R_u(x))^{\circ} \ \ \forall x\in \mathfrak Q^{-1}(\alpha)\subset \T_u.
\end{align}
where $(R_u(x))^\circ$ denotes the relative interior of the closed set $R_u(x)$.

The following is \cite[Theorem 3.5]{cav-mon-lapl-18}.
\begin{theorem}[Factor measures inherit curvature-dimension bounds]\label{th:1dlocalisation} Let $K\in \R$, $N\in (1,\infty)$ and let $(X,\de,\m)$ be essentially nonbranching  and $MCP(K,N)$   with $\supp\m=X$.
For any $1$-Lipschitz function $u:X\rightarrow \R$, {\red let $\{\m_{\alpha}\}_{\alpha\in Q}$ denoted the disintegration of $\m|_{\T_u}$ 
from Theorem \ref{T:CM disintegration} which is strongly consistent with the quotient map $\mathfrak Q:\T_u \rightarrow Q$.

Then} there exists $\tilde Q$ such that $\mathfrak q(Q\backslash \tilde Q)=0$ and $\forall \alpha\in \tilde Q$, $\m_{\alpha}$ is a Radon measure with $d\!\m_{\alpha}=h_{\alpha}d\mathcal{H}^1|_{X_{\alpha}}$ and $(X_{\alpha}, d, \m_{\alpha})$ satisfies $MCP(K,N)$. If $(X,\de,\m)$ satisfies the condition $\CD(K,N)$, then $(X, d, \m_\alpha)$ satisfies the condition $\CD(K,N)$ as well.
%More precisely, for all $\alpha\in \tilde Q$ it follows that
%\begin{align}\label{kuconcave}
%h_{\alpha}(\gamma_t)^{\frac{1}{N-1}}\geq \sigma_{K/N-1}^{(1-t)}(|\dot\gamma|)h_{\alpha}(\gamma_0)^{\frac{1}{N-1}}+\sigma_{K/N-1}^{(t)}(|\dot\gamma|)h_{\alpha}(\gamma_1)^{\frac{1}{N-1}}
%\end{align}
%for every {affine map} $\gamma:[0,1]\rightarrow (a(X_\alpha),b(X_\alpha))$.
\end{theorem}
%
%\begin{remark}[{\red Semiconcave densities on needles}] \label{rem:kuconcave} 
\begin{remark}\label{rem:kuconcave}
The theorem
yields that $h_{\alpha}$ is locally Lipschitz continuous on $(a(X_\alpha),b(X_\alpha))$ \cite[Section 4]{cavmon}.
In particular $h_\alpha$ is differentiable for $\mathcal L^1$-a.e. $r\in (a(X_\alpha), b(X_\alpha))$ and
\begin{align*}
\frac{\overline{d^+}}{dr} h_\alpha(r) \!=\! \limsup_{h\downarrow 0} \frac{h_\alpha(r+h)-h_\alpha(r)}{h},\ \ 
\frac{\overline{d^-}}{dr} h_\alpha(r)\! =\! \limsup_{h\uparrow  0} \frac{h_\alpha(r+h)-h_\alpha(r)}{h}
%\\
%\frac{{d^+}}{\underline{dr}} h_\alpha(r) = \liminf_{h\downarrow 0} \frac{h_\alpha(r+h)-h_\alpha(r)}{h}& \ \ \& \ \ 
%\frac{{d^-}}{\underline{dr}} h_\alpha(r) = \liminf_{h\uparrow  0} \frac{h_\alpha(r+h)-h_\alpha(r)}{h}
\end{align*}
both exist in $\R$ for all $r\in (a(X_\alpha), b(X_\alpha))$.  
The Bishop-Gromov volume monotonicity implies that $h_\alpha$ can be extended to a continuous function on $[a(X_\alpha),b(X_\alpha)]$ \cite[Remark 2.14]{cav-mon-lapl-18}. 
%
%Then \eqref{kuconcave} holds for every affine map $\gamma:[0,1]\rightarrow [a(X_\alpha),b(X_\alpha)]$.
%We set $\left(h_{\alpha}\circ \gamma_{\alpha}(r)\right)\cdot1_{[a(X_{\alpha}),b(X_{\alpha})]}=h_{\alpha}(r)$
%and consider $h_{\alpha}$ as function that is defined everywhere on $\R$.
We  consider $\frac{d}{dr}h_{\alpha}: X_{\alpha}\rightarrow \R$ defined a.e.\ via
$\frac{d}{dr}(h_{\alpha}\circ\gamma_{\alpha})(r)=:\frac{d}{dr}h_{\alpha}( \gamma_\alpha(r))$.
\end{remark}
%\marginpar{For me,  the red notation would be less ambiguous;  the blue looks like $\liminf$ rather than left derivative}
%It is standard knowledge about semi-concave functions that the derivatives from the {\red right} and from the {\red left}
%\begin{align*}
%\frac{d^{+}}{dr} h_\alpha(r)= \lim_{t\downarrow 0} \frac{h_{\alpha}(r+t)-h_\alpha(r)}{t},\ \
%\frac{d^-}{dr}h_\alpha(r)= \lim_{t\uparrow 0}\frac{h_{\alpha}(r+t)-h_\alpha(r)}{t}
% \end{align*}
%exist for  $r\in [a(X_\alpha), b(X_\alpha))$ and $r\in (a(X_\alpha), b(X_\alpha)]$ respectively. Moreover, we  set $\frac{d^{+}}{dr}h_\alpha= - \infty$ in $b(X_\alpha)$ and $\frac{d^-}{dr}h_\alpha=\infty$ in  $a(X_\alpha)$. }
%%\end{remark}
\begin{remark}[Generic geodesics]\label{def:dagger}
\label{dagger}
We set $Q^\dagger:= \tilde Q \cap \hat Q$,  where $ \tilde Q$  and $\hat Q$ index the transport rays identified between Lemma \ref{somelemma}
and Theorem \ref{th:1dlocalisation}.
 Then, $\mathfrak q(Q\backslash Q^\dagger)=0$ and  for every $\alpha \in Q^\dagger$ the space $(X, \de, h_\alpha \mathcal H^1)$ is $MCP(K,N)$ (or $\CD(K,N)$) and \eqref{somehow} holds. We also set $\mathfrak Q^{-1}(Q^{\dagger})=:\T_u^{\dagger}\subset \T_u$ and $\bigcup_{x\in \T_u^{\dagger}} R_u(x)=:\T_{u,e}^{\dagger}\subset \T_{u,e}$.
\end{remark}
\section{Notions of synthetic lower mean curvature bounds}\label{subsec:meancurvature} 
%\marginpar{I purged $d^*_\Omega$, which was not used elsewhere; is $d_{\red X \setminus \Omega}$ more legible?}

%Let $(X,d,\m)$ be a metric measure space as in Theorem \ref{th:1dlocalisation}.
Let $(X, \de, \m)$ be an $\RCD$ space with $\supp\m=X$ and let $\Omega\subset X$ be an open subset such that $\m(\partial \Omega)=0$.  We set $S:=\partial \Omega=\overline \Omega\backslash \Omega$ and  $\Omega^c:= X\backslash \Omega$. Since $\m(S)=0$, it holds $\partial \Omega^c=S$. The distance function $\de_{\Omega^c}: X\rightarrow \mathbb{R}$ is given by
\begin{align*}
{\red \de_{{\Omega}^c}(x) :=} \inf_{y\in {\Omega}^c} \de(x,y).
\end{align*}
%Let us also define $d_{\Omega}^*:=d_{{ \Omega^c}}$. 
The signed distance function $\de_S$ for $S$ is  given by
\begin{align*}
\de_S:=d_{{\overline\Omega}}- \de_{{\Omega^{c}}}: X\rightarrow \mathbb{R}.
\end{align*}
It follows that $\de_S(x)=0$ if and only if $x\in S$, and  $\de_S\leq 0$ if $x\in \Omega$ and $\de_S\geq 0$ if $x\in \Omega^c$.
It is clear that $\de_S|_{\Omega}= -\de_{\Omega^{ c}}$ and $\de_S|_{\Omega^c}=\de_\Omega$.
Setting $v=\de_S$ we can also write
$$
\de_S(x)= \sign(v(x))\de(\{v=0\},x), \forall x\in X.
$$
Since $(X, \de)$ is a proper geodesic space, $\de_S$ is $1$-Lipschitz \cite[Remark 8.4, Remark 8.5]{cav-mon-lapl-18}. 
%Let $\Omega^\circ$ denote the topological interior of $\Omega$.

%
%
%
Let $\mathcal{T}_{\de_S,e}$ be the transport set of $\de_S$ with end- and branching points. We have $\mathcal{T}_{\de_S, e}\supset X\backslash S$. In particular, we have $\m(X\backslash \T_{\de_S})=0$ by Lemma \ref{somelemma} and $\m(S)=0$.
Therefore, by Theorem \ref{th:1dlocalisation} the $1$-Lipschitz function $\de_S$ induces a partition $\left\{X_{\alpha}\right\}_{\alpha\in Q}$ of $X$ up to a set of measure zero for a measurable quotient space $Q$, and a disintegration $\{\m_{\alpha}\}_{\alpha\in Q}$ that is strongly consistent with the partition. The subset $X_{\alpha}$, $\alpha \in Q$, is the image of a distance preserving map $\gamma_{\alpha}:I_\alpha\rightarrow X$ for an interval $I_\alpha\subset \mathbb R$ with $\overline{I}_\alpha= [a(X_\alpha), b(X_\alpha)]\ni 0$.

We consider $Q^\dagger\subset Q$ as in Remark \ref{dagger}. One has the representation
\begin{equation*}
\m(B)=\int_{Q} \m_{\alpha}(B) d\mathfrak q(\alpha)=\int_{Q^{\dagger}} \int_{\gamma_{\alpha}^{-1}(B)} h_{\alpha}(r) dr d\mathfrak q(\alpha)
\end{equation*}
for all Borel subsets $B \subset X$.
For a transport ray $X_{\alpha}$ one has
{$\de_S(\gamma_\alpha(b(X_{\alpha})))\geq 0$ and
$\de_S(\gamma_\alpha(a(X_{\alpha})))\leq 0$} (for instance compare with \cite[Remark 4.12]{cav-mon-lapl-18}).

Let us recall another result of Cavalletti-Mondino{:}

\begin{theorem}[{Laplacian of  signed distance functions} {\cite[Corollary 4.16]{cav-mon-lapl-18}}]\label{thm:cm}
Let $(X,\de,\m)$ be a $\CD(K,N)$ space, and $\Omega$ and $S=\partial \Omega$ as above. Then
$\de_S|_{X\backslash S}\in D({\bf \Delta}, X\backslash S)$, and one element of ${\bf \Delta}_{X\backslash S}(d_S|_{X\backslash S})$ that we also denote with ${\bf \Delta }_{X\backslash S}(\de_S|_{X\backslash S})$ is the Radon functional on $X\backslash S$ given by the representation formula
\begin{align*}
{\bf \Delta}_{X\backslash S}(\de_S|_{X\backslash S})&= (\log h_{\alpha})'\m|_{X\backslash S}\\
&\ \ \  \ \ \ +\int_Q ( h_{\alpha}\delta_{a(X_{\alpha})\cap \{\de_S {\red <} 0\}} - h_{\alpha}\delta_{b(X_{\alpha})\cap \{\de_S {\red >}0\}} ) d\mathfrak q(\alpha).
\end{align*}
The Radon functional ${\bf \Delta}_{X\backslash S}( \de_S|_{X\backslash S})$ can be represented as the difference of two measures $[{\bf \Delta} _{X\backslash S}(\de_S|_{X\backslash S})]^+$ and $[{\bf \Delta}_{X\backslash S} (d_S|_{X\backslash S})]^-$ such that
\begin{align*}
[{\bf \Delta}_{X\backslash S}( d_S|_{X\backslash S})]^+_{abs} - [{\bf \Delta}_{X\backslash S}(d_S|_{X\backslash S})]^-_{abs} =  (\log h_{\alpha})' \ \ \m\mbox{-a.e.}
\end{align*}
%where $[{\bf \Delta}_{X\backslash S}(\de_S|_{X\backslash S})]^{\pm}_{reg}$ denotes the $\m$-absolutely continuous part in the Lebesgue decomposition of $[{\bf \Delta}_{X\backslash S}( \de_S|_{X\backslash S})]^{\pm}$. 
In particular, $ -(\log h_{\alpha})'$ coincides with a measurable function $\m$-a.e.
\end{theorem}
%\marginpar{\blue redundant minus sign?}
%

\begin{remark}[Measurability and zero-level selection]\label{R:zero-level selection}
It is easy to see that $A:=\mathfrak Q^{-1}(\mathfrak Q(S\cap \T_{d_S}))\subset \T_{d_S}$ is a measurable subset.
The {\red \it reach}
$A\subset \T_{\de_S}$ is defined such that  $\forall \alpha \in \mathfrak Q(A)$ we have $X_\alpha\cap S =\{\gamma(t_\alpha)\}\neq \emptyset$ for a unique $t_\alpha\in I_{\alpha}$. Then, the map ${\hat s}: \gamma(t)\in A \mapsto \gamma(t_\alpha)\in S\cap \T_{d_S}$ is a measurable section {(i.e.~selection)} on $A\subset \T_{d_S}$, {\red and} one can  identify the measurable set $\mathfrak Q(A)\subset Q$ with $A\cap S$ and %one 
{\red can} parameterize $\gamma_\alpha$ such that $t_\alpha=0$.

This measurable section $\hat s$ on $A$ is fixed for the rest of the paper.
The {\red reach} $A$ is the union of all  disjoint needles that {\red intersect} with $\partial \Omega$ {\red  -- eventually in $a(X_\alpha)$ (or in $b(X_\alpha)$) provided $a(X_\alpha)$ (respectively $b(X_\alpha)$) belongs already to $I_\alpha$}.  We shall also define {\red the \it inner reach} $B_{in}$ 
as the union of all needles disjoint from
${\red {\Omega^c}}$ and {\red the \it outer reach} $B_{out}$ as the union of all needles disjoint from ${\red \overline\Omega}$.  The superscript $\dagger$ will be used to indicate intersection with $\mathcal T^{\dagger}_{\de_S}$.
Thus
\begin{align*}
\mbox{$
{A}\cap \mathfrak \T_{\de_S}^{\dagger}=: A^{\dagger} \ \mbox{ and }\ \bigcup_{x\in A^{\dagger}}R_{\de_S}(x)=:A_e^{\dagger}.$}
\end{align*}
The sets $A^\dagger$ and $A_e^\dagger$ are measurable, and also
\begin{align}\label{Binout}
\mbox{$
B_{in}^\dagger:=\Omega^\circ \cap \T^\dagger_{\de_S} \backslash A^\dagger\subset \T^\dagger_{\de_S}$ \ and \ $B^\dagger_{out}:=\Omega^c \cap \T^\dagger_{\de_S} \backslash A^\dagger\subset \T_{\de_S}$}
\end{align}
as well as $\bigcup_{x\in B^\dagger_{out}} R_{\de_S}(x)=:B^\dagger_{out, e}$ and $\bigcup_{x\in B^\dagger_{in}}R_{\de_S}(x)=:B^\dagger_{in,e}$
are measurable.
The map $\alpha \in  \mathfrak Q(A^\dagger)\mapsto h_\alpha(0)\in \R$ is  measurable (see \cite[Proposition 10.4]{cavmil}).
\end{remark}
%
%
%
%\begin{definition}[Surface measures]\label{def:surfacemeasure}
%Taking $S=\partial \Omega$ as above, we use the disintegration of Remark \ref{R:zero-level selection} to define the {\it surface measure} $\m_S$ via
%$$
%\int\phi(x)d\m_S(x):=\int_{\mathfrak Q(A^{\dagger})} \phi(\gamma_\alpha(0)) h_{\alpha}(0) d\mathfrak q(\alpha)
%$$
%for  any bounded and continuous function $\phi:X\rightarrow \R$.
%%
%%
%%\begin{remark}
% That is, $\m_S$ is the pushforward of the measure {$h_\alpha(0) {\red d}\mathfrak q(\alpha)|_{\mathfrak Q(A^{\dagger})}$} under the map ${\red \hat s}:\gamma\in \mathfrak Q(A^{\dagger})\mapsto \gamma(0) \red \in S$.
%% 
%%
% %\end{remark}
%\end{definition}

{
 \begin{remark}[{\red Surface measure via} ray maps] \label{rem:surmea}Let us  briefly explain the previous definition from the viewpoint of the ray map \cite[Definition 3.6]{cavmon}
{ or its precursor from the smooth setting \cite{FeldmanMcCann02}.}
For the definition we fix a  measurable {\red extension} $ s_0:\mathcal T_{\de_S}\rightarrow \mathcal  T_{\de_S}$ such that $ s_0|_{A^{\dagger}}=\hat s$ as in Remark \ref{R:zero-level selection}.  As was explained in Subsection \ref{subsec:1D} such a {section} allows us to identify the quotient space $Q$ with a Borel subset in $X$ up to a {\red set of $\mathfrak q$-}measure $0$.
 Following \cite[Definition3.6]{cavmon} we define the ray map 
 \begin{align*}
 g:\mathcal V\subset \mathfrak Q(A\cup B_{in})\times {\red (-\infty,0]} \rightarrow \Omega
 \end{align*}
{\red into $\Omega$ and its domain $\mathcal V$} via its graph
%  \marginpar{Why the reparameterization which distinguishes points in the reach from those in the inner reach? Isn't this ultimately irrelevant for us?}
 \begin{align*}
 \mbox{graph}(g)&=\{ (\alpha,t,x)\in \mathfrak Q(A)\times\R\times \Omega:  x\in X_\alpha, -\de(x,\alpha)=t\}\\
 &\hspace{0.5cm}\cup  \{(\alpha,t,x)\in \mathfrak Q(B_{in}) \times \R\times \Omega: x\in X_\alpha, -\de(x, \gamma_\alpha(b(X_\alpha)))=t\}.
 \end{align*}
 This is exactly the ray map as in \cite{cavmon} up to a reparametrisation for $\alpha\in \mathfrak Q(B_{in})$.
{Note that $g(\alpha,0)=\gamma_\alpha(0)=\alpha$ and $g(\alpha,t)=\gamma_\alpha(t)$ if $\alpha\in \mathfrak Q(A)$ but $\gamma_\alpha(t+\de(b(X_\alpha),\alpha))=g(\alpha,t)$ for $\alpha\in \mathfrak Q(B_{in})$.}
 Then the disintegration for a non-negative $\phi\in  C_b({\red \Omega})$ takes the form
% \marginpar{I'm confused.  $g$ takes values only on $\Omega$.  Shouldn't we restrict $\phi$ to vanish outside $\Omega$, i.e., $\blue \phi \in C_b(\Omega)$?}
 \begin{align*}
 \int_{\red \Omega} \phi\thinspace { \dm}  = \int_Q \int_{\mathcal V_\alpha} \phi\circ g(\alpha,t) h_\alpha\circ g(\alpha,t) d \mathcal L^1{ (t)} d\mathfrak q(\alpha)
 \end{align*}
where $\mathcal V_\alpha = P_2(\mathcal V \cap \{\alpha\}\times \R)\subset \R$ and $P_2(\alpha,t)=t$.  With Fubini's theorem the right hand side is
 \begin{align*}
 \int_{\mathcal V} \phi \circ g(\alpha,t) h_\alpha \circ g(\alpha,t) d(\mathfrak q\otimes \mathcal L^1)(\alpha,t)= \int \int_{\mathcal V_t} \phi  \circ g(\alpha, t) h_\alpha\circ g(\alpha, t) d\mathfrak q(\alpha) d\mathcal L^1(t)
 \end{align*}
 where $\mathcal V_t= P_1(\mathcal V \cap Q\times \{t\})\subset Q$ and $P_1(\alpha,t)=\alpha$. In particular, for $\mathcal L^1$-a.e. $t\in \R$ the set $\mathcal V_t\subset Q$ and  the map $\alpha \mapsto h_{\alpha}\circ g(\alpha,t)$ are measurable.
 Hence, for $\mathcal L^1$-a.e. $t\in \R$ we define ${\red d}\mathfrak p_t{\red (\alpha)}= h_\alpha \circ g(\alpha,t) {\red d} \mathfrak q|_{\mathcal V_t}{\red(\alpha)}$ on $Q$. Then the disintegration takes the form
 \begin{align*}
 \m|_{\Omega}=\m|_{\Omega\cap \mathcal T_{d_S}}= \int (g(\cdot,t)_{\#} \mathfrak p_t) dt.
 \end{align*}
 Note that $\mathcal V_0= \mathcal V\cap Q\times \{0\}= \mathfrak Q(A)\dot \cup \mathfrak Q(B_{in})$ is measurable, one has $\mathcal V_t\subset \mathcal V_0$, $t<0$, and that $\alpha\in \mathcal V_0 \mapsto \lim_{t\uparrow 0} h_\alpha\circ g(\alpha, t)= h_\alpha \circ g(\alpha,0)$ is  measurable.  Hence, we set $d\mathfrak p_0(\alpha) = h_\alpha\circ g(\alpha, t) d\mathfrak q|_{\mathcal V_0}(\alpha)$. 
 
%We can consider the pushforward
%$
%\m_{S_0}= g(\cdot,0)_{\#} \mathfrak p_0.
%$
%as a surface measure on $S$. 
%We see that $\m_{S_0}$ {\red may be} concentrated on {\red a} larger set than $\m_{S}$ but by  construction one recognizes that $\m_{S}=\m_{S_0}|_{A^{\dagger}}$.
\end{remark}}

\begin{definition}[Backward mean curvature bounded below] Let $(X,d,\m)$ be essentially nonbranching and $MCP(K,N)$ for $K\in \R$ and $N\in (1,\infty)$.
%Recall the family of measures {\red $\{ \mathfrak p_t \}_{t \in \blue (-\infty,0]}$ on $Q$ given by}
%${\red d}\mathfrak p_t({\red \alpha})= h_\alpha\circ g(\alpha,t) {\red d}\mathfrak q|_{\mathcal V_t}{\red (\alpha)}$ %$t\in [0,\infty)$ on $Q$, 
%that we introduced in Remark \ref{rem:surmea}. Recall  that $Q$ is constructed as a Borel subset of $X$ and $(\alpha,t)\mapsto g(\alpha,t)$ is the ray map 
%constructed in {\red that remark.}
%% Remark \ref{rem:surmea}.

Then $S=\partial \Omega$ has \emph{backward mean curvature bounded from below by $H\in \R$} if the measure {$\mathfrak p_0$ is a Radon measure}, {$h_\alpha\circ g(\alpha,0)>0$ for $\mathfrak q$-a.e. $\alpha \in Q$} and
\begin{align*}{
\overline{\frac{d^-}{dt}} \Big|_{t=0} \int_\Y d\mathfrak{p}_t 
:= \limsup_{h\uparrow 0} \frac{1}{h}\left( \int_\Y d\mathfrak p_h- \int_\Y d\mathfrak p_0\right)\geq H\int_\Y d\mathfrak p_0}
\end{align*}
for any {\red bounded measurable subset $Y\subset Q$.
%for all continuous {\blue and compactly supported functions} $\phi: X\rightarrow [0,\infty)$.
Moreover,  $S$ has \emph{backward-lower} mean curvature bounded from below by $H$ if the same inequality holds when $\limsup$ is replaced by $\liminf$.
%If, in addition, the $\limsup$ and $\liminf$ agree, we say $S$ has backwards mean curvature bounded below by $H$.
%Heuristically, these limits quantify the relative rate of change of surface area of the level sets of $d_S$, as when $Y=Q \subset X$ is bounded.
}
\end{definition}

\begin{remark}
Since it is not assumed that $(\mathfrak p_t)_{t>0}$ is a Radon measure, $\int_\Y d\mathfrak p_t$ can be infinite.
\end{remark}
\begin{proposition}[Rescaling]
Let $(X,\de,\m)$ be $MCP(K,N)$ and let $\Omega\subset X$ with backward mean curvature bounded below by $H\in \R$ in $X$. Define $(\tilde X, \tilde \de,\tilde \m)$ with $\tilde X= X$, $\tilde \m=\m$ and $\tilde \de= \epsilon \de$.  Then $\tilde X$ satisfies $MCP(\frac{1}{\epsilon^2}K, N)$ and  $\Omega$ has mean curvature bounded from below by $\frac{1}{\epsilon}H$ in $\tilde X$. 
\end{proposition}
\begin{proof}
The first claim is  known. For the second claim observe that $(x,y)\in \Gamma_{\de_{S}}$, satisfies 
\begin{align*}
\tilde \de_S(y) - \tilde \de_S(x) = \epsilon\left( \de_S(y) - \de_S(x)\right) = \epsilon \de(x,y) = \tilde \de (x,y). 
\end{align*}
Hence given transport geodesic $\gamma_\alpha$ w.r.t. $\de_S$ we have $r\in [\epsilon a(X_\alpha), \epsilon b(X_\alpha)] \mapsto \gamma(\frac{1}{\epsilon} r)$ is transport geodesic w.r.t. $\tilde \de_S$. This implies that $\partial \Omega = S$ has backward mean curvature bounded below by $\frac{1}{\epsilon} H$. 
\end{proof}

\begin{lemma}\label{lemma:localisation}
Let $(X,\de, \m)$ be an essentially non-branching $MCP(K,N)$ space with $K\in \R$, $N\in (1,\infty)$, and let $\Omega\subset X$ such that $S=\partial \Omega$ has backward mean curvature bounded from below $H$.
Then 
\begin{align}\label{ineq:mean}
\overline{\frac{d^-}{dr}}\Big|_{r=0} h_\alpha\circ g(\alpha,r)\geq H 
 h_\alpha(g(\alpha,0))
\end{align}
for $\mathfrak q$-a.e.  $\alpha \in \mathcal V_0= \mathfrak Q(A^{\dagger}\cup B_{in}^{\dagger}).$ 

If $\frak{p}_0$ is a Radon measure, $h_\alpha\circ g(\alpha, 0)>0$ for $\mathfrak q$-a.e. $\alpha\in Q$ and 
\begin{align}\label{ineq:mean2}
\underline{\frac{d^-}{dr}}\Big|_{r=0} h_\alpha\circ g(\alpha,r)\geq H 
 h_\alpha(g(\alpha,0)) \ \ \mbox{ for } \mathfrak q\mbox{-a.e. }  \alpha\in \mathcal V_0
\end{align}
then $S$ has backward-lower (hence backward) mean curvature bounded from below by $H$.

If $(X,\de, \m)$ is a $\CD(K,N)$ space, then \eqref{ineq:mean} and \eqref{ineq:mean2} become
\begin{align*}
{\frac{d^-}{dr}}\Big|_{r=0} h_\alpha\circ g(\alpha,r)\geq H 
 h_\alpha(g(\alpha,0))
\end{align*}
and hence, backward and backward-lower mean curvature bounded from below are equivalent. $\frac{d^-}{dr}$ is the left derivative.
\end{lemma}
\begin{proof} We start with the first claim.
For $t<0$ and a bounded measurable set $Y\subset Q$ we write
\begin{align*}
 & \int_\Y d\mathfrak p_t- \int_Y d\mathfrak p_0\\
& =   \int_\Y h_\alpha(g(\alpha,0))^{-1}\left(1_{\mathcal V_t}(\alpha) h_\alpha\circ g(\alpha,t)  - 1_{\mathcal V_0}(\alpha) h_\alpha\circ g(\alpha,0)\right) d\mathfrak p_0(\alpha).
 \end{align*}
{\red There exists  $Q^*\subset Q^\dagger$ {with  $\mathfrak q[Q^\dagger\setminus Q^*]=0$} such that the map $\mathcal M: \alpha\in Q^* \mapsto -a(X_\alpha)$ is measurable (compare with the proof of Theorem 7.10 in \cite{cavmil} or Remark 3.4 in \cite{kks}). 

Then, we consider measurable sets $Q_{\red m}=\mathcal M^{-1}([\frac{1}{m},m])$ for $m\in \mathbb N$. It holds $\bigcup_{m\in \mathbb N_0}Q_{\red m}=Q^*\cap \mathfrak Q(A^{\dagger}\cup B_{in}^\dagger)$.  {From  \cite[Appendix A2]{cavmil} we see}
\begin{align*}
&{h_\alpha\circ g(\alpha,0)}^{-1}\frac{1}{r}(h_\alpha\circ g(\alpha,r)  -  h_\alpha\circ g(\alpha,0))\\
&\ \ \ \ \ \ \leq (N-1)\frac{\cos_{-|K|/(N-1)}(-a(X_\alpha))}{\sin_{-|K|/(N-1)}(-a(X_\alpha))}\leq C(K,N,m) 
\end{align*}$\forall r\in (a(X_\alpha),0), \ \forall \alpha\in Q_{\red m}$.}
Thus we can apply Fatou's lemma: 
\begin{align*}
&H\int_{\Y\cap Q_{\red m} \cap \mathcal V_0} h_\alpha\circ g(\alpha,0) d\mathfrak q(\alpha)\\
&=H\int_{\Y\cap Q_{\red m}} d\mathfrak p_0(\alpha)\\
&\leq   \int_{\Y\cap Q_{\red m}} \limsup_{t\uparrow 0}\frac{1}{t} \left(1_{\mathcal V_t}(\alpha) h_\alpha\circ g(\alpha,t) - 1_{\mathcal V_0}(\alpha) h_\alpha\circ g(\alpha,0)\right) d\mathfrak q(\alpha)
 \\
 &\leq \int_{\Y\cap Q_{\red m}} \limsup_{t\uparrow 0}\frac{1}{t} \left(1_{\mathcal V_t\cap \mathcal V_0}(\alpha) h_\alpha \circ g(\alpha,t)  - 1_{\mathcal V_0}(\alpha) h_\alpha\circ g(\alpha,0)\right) d\mathfrak q(\alpha)
 \\
 & =\int_{ \Y\cap Q_{\red m}\cap \mathcal V_0} \overline{\frac{d^-}{dt}}\Big|_{t=0}  h_\alpha \circ g(\alpha,t) d\mathfrak q(\alpha)
 \end{align*}
 for any bounded {\red measurable  set $Y \subset Q$.} Fatou's lemma was used in the first inequality together with the backward lower mean curvature bound. 
 It  follows that
 \begin{align}\label{inequ:nini}{
 H h_\alpha\circ g(\alpha,0)\leq \overline{\frac{d^-}{dt}}|_{t=0}h_\alpha\circ g(\alpha,t) \mbox{ for }\mathfrak q\mbox{-a.e. }\alpha\in \mathcal V_0.}
 \end{align}
The second claim follows similarly with Fatou's Lemma ($\liminf$ version).
%{\blue We assume{\red d} $h_\alpha\circ g(\alpha,0)>0$ for $\mathfrak q$-almost every $\alpha\in \mathcal V_0$.} Hence the claim follows.
\end{proof}
\begin{theorem}\label{th:laplacecomparison}
Let $X$ be an essentially non-branching $\CD(K,N)$  space with $K\in\R$, $N\in (1, \infty)$, and let $\Omega\subset X$ be open. Let $u=\de_{S}|_{\Omega}=-\de_{\Omega^c}|_{\Omega}$. Assume $\mathfrak p_0$ is a Radon measure and $h_\alpha\circ g(\alpha,0)>0$ for $\mathfrak q$-a.e. $\alpha\in Q$.  

Then $\partial \Omega$ has backward mean curvature bounded from below by $H\in \R$ if and only if 
\begin{align}\label{ineq:lapl0}
{\bf \Delta}_{\Omega} u\geq - (N-1) \frac{s'_{\frac{K}{N-1}, \frac{H}{N-1}}(-u)}{s_{\frac{K}{N-1}, \frac{H}{N-1}}(-u)}\m|_{\Omega}.
\end{align}
In particular, if  $K\leq 0$ and $H=\pm\sqrt{{|K|}(N-1)}$, then \eqref{ineq:lapl0} becomes
\begin{align}\label{llll}
{\bf \Delta}_{\Omega} u\geq  \mp (N-1) \sqrt{\frac{|K|}{N-1}} \m|_{\Omega}.
\end{align}
\end{theorem}
\begin{proof} "$\Rightarrow$": {The proof of  inequality \eqref{ineq:lapl0} already appears in  \cite{bkmw}. For completeness we will provide details.} 
Recall
\begin{lemma}[Riccati comparison]\label{lem:riccati}
Let $u:[0,b]\rightarrow \R$ be {\red non-negative and} continuous such that $u''+\kappa u\leq 0$ in {the} distributional sense, $u(0)=1$ and ${\frac{d^+}{dr}}u(0)\leq- d$.
Let $v:[0,\bar b]\rightarrow \R$ be the maximal {\red non-negative} solution  of $v''+\kappa {\red v}  = 0$ with $v(0)=1$ and $v'(0)=-d$. That is, 
{\red $v=s_{\kappa,d}$ from \eqref{equ:ks}.} %that was defined in equation \eqref{equ:ks} before Theorem \ref{T:main}.
Then $\bar b\geq b$ and  ${\frac{d^+}{dt}} \log u\leq (\log v)' $ on $[0,b)$.
\end{lemma}
Let  $\{X_\alpha\}_{\alpha\in Q}$ be the decomposition of $\T_u$ and $\int \m_\alpha d\mathfrak q(\alpha)$ be the disintegration of $\m$ given by Theorem \ref{T:CM disintegration}
and Remark~\ref{R:zero-level selection}. Recall that $m_\alpha=h_\alpha \mathcal H^1$ for $\mathfrak q$-a.e. $\alpha\in Q$.
We consider $Q^\dagger\subset Q$ that has  full $\mathfrak q$-measure as defined in Remark \ref{dagger}.  For every $\alpha\in Q^{\dagger}$ we have that 
$\m_\alpha = h_\alpha \mathcal H^1$, 
$X_{\alpha,e}=\overline{X}_\alpha$ and $h_\alpha$ is continuous on $[a(X_\alpha),0]$ by Remark \ref{rem:kuconcave} and satisfies
\begin{align}\label{inequality}
(h_\alpha^{\frac{1}{N-1}})'' + \frac{K}{N-1} h_\alpha^{\frac{1}{N-1}}\leq 0 \mbox{ on } (a(X_\alpha),0) \ \forall \alpha\in Q^{\dagger},
\end{align}
in the distributional sense. As usual we write $h_\alpha=h_\alpha\circ \gamma_\alpha$. We also have the properties of $h_\alpha$ as discussed in Remark \ref{rem:kuconcave}.  
By the definition of backward  mean curvature bounded from below  it holds $h_\alpha(r)>0$ for $\mathfrak q$-a.e. $\alpha$.

The function $r\in [0,-a(X_\alpha)]\mapsto \tilde h_\alpha(r) := h_\alpha(-r)$ is also continuous and \eqref{inequality}  still holds on $(0,-a(X_\alpha))$. Lemma \ref{lemma:localisation} implies
\begin{align*}
\frac{d^+}{dr}\Big|_{r=0}\tilde{h}\circ g(\alpha, r) \leq - H \tilde h\circ g(\alpha,0).
\end{align*}
and hence with Lemma \ref{lem:riccati}
\begin{align*}
(\log \tilde h_\alpha)'(r)\leq \left(\log \left(s_{\frac{K}{N-1}, \frac{H}{N-1}}(r)\right)^{N-1}\right)'.
\end{align*}
By Theorem \ref{thm:cm} we also have 
\begin{align*}
{\bf \Delta}_{\Omega} u= (\log h_{\alpha})'\m|_{\Omega}+
\int_Q  h_{\alpha}\delta_{a(X_{\alpha})\cap \Omega^{\circ}} d\mathfrak q(\alpha)
%- 
%h_{\alpha}\delta_{b(X_{\alpha})\cap \{d_S\red >0\}} }  .
\geq  (\log h_{\alpha})'\m|_{\Omega} = - (\log \tilde h_{\alpha})'\m|_{\Omega} 
\end{align*}
where we also used Lemma \ref{lem:negsin} from the next section. 
This yields the estimate for ${\bf \Delta}_{\Omega} u$. 

For the  estimate  \eqref{llll} we recall that 
\begin{align*}
\frac{s'_{\frac{K}{N-1}, \frac{H}{N-1}}}{s_{\frac{K}{N-1}, \frac{H}{N-1}}}= \frac{-\left(\frac{K}{N-1}\right)\cdot\sin_{\frac{K}{N-1}} - \left(\frac{H}{N-1}\right)\cdot\cos_{\frac{K}{N-1}}}{\cos_{\frac{K}{N-1}} - \left(\frac{H}{N-1}\right)\cdot \sin_{\frac{K}{N-1}}}.
\end{align*}
Using the value $-H^2 = K(N-1) \ \Leftrightarrow \  H=\pm \sqrt{|K|(N-1)}$ $\Leftrightarrow$ $\frac{H}{N-1}=\pm \sqrt{\frac{|K|}{N-1}}$, it follows
\begin{align*}
\frac{s'_{\frac{K}{N-1}, \pm\sqrt{\frac{|K|}{N-1}}}}{s_{\frac{K}{N-1},\pm\sqrt{\frac{|K|}{N-1}}}}&= \frac{\frac{|K|}{N-1}\sin_{\frac{K}{N-1}} \mp{\scriptscriptstyle \sqrt{\frac{|K|}{N-1}}}\cos_{\frac{K}{N-1}}}{\cos_{\frac{K}{N-1}} \mp {\scriptscriptstyle \sqrt{\frac{|K|}{N-1}}} \sin_{\frac{K}{N-1}}} \\
&= \mp  \sqrt{\frac{|K|}{N-1}}\frac{\mp {\scriptscriptstyle\sqrt{\frac{|K|}{N-1}}}{\sin_{\frac{K}{N-1}} + \cos_{\frac{K}{N-1}}}}{\cos_{\frac{K}{N-1}} \mp {\scriptscriptstyle \sqrt{\frac{|K|}{N-1}}} \sin_{\frac{K}{N-1}}}=\mp \sqrt{\frac{|K|}{N-1}}.
\end{align*}
This proves the claim.
\smallskip

"$\Leftarrow$":  The assumption and Theorem \ref{thm:cm} imply that $\mathfrak q$-a.e. $\alpha\in Q$ there exists a sequence $(r_n)_{n\in \N}$ in $(0, -a(X_\alpha))$ such that $r_n\downarrow 0$ and 
\begin{align} \label{zzx}
\frac{d}{dr} \log h_\alpha \circ g(\alpha,-r_n) \geq  -(N-1) \frac{s_{\frac{K}{N-1}, \frac{H}{N-1}}'}{s_{\frac{K}{N-1}, \frac{H}{N-1}}}(r_n). \end{align}
%for $dr\mbox{-a.e. } r_n\in (0,-a(X_\alpha)) \mbox{ and } \mathfrak q\mbox{-a.e. }\alpha\in Q.$
Since $h_\alpha$ is a semi-concave function for $\mathfrak q$-a.e. $\alpha\in Q$ on $[a(X_\alpha),b(X_\alpha)]$, its right-derivative is right-continuous on $[0,-a(X_\alpha))$.  In particular $
\frac{d}{dr} \log h_\alpha \circ g(\alpha,-r_n) \rightarrow \frac{d^-}{dr} \log h_{\alpha}(r)\Big|_0$ for $r_n\downarrow 0$.  On the other hand, the right hand side of \eqref{zzx} converges to $H$ for $r_n\downarrow 0$. 
One obtains
\begin{align*}
{\frac{d^-}{dr}}\Big|_{r=0} h_\alpha\circ g(\alpha,r)\geq H 
 h_\alpha(g(\alpha,0))\ 
\mbox{ for $\mathfrak q$-a.e. $\alpha\in Q$. }
\end{align*}Hence, by Lemma \ref{lemma:localisation} $S$ has backward mean curvature bounded from below.
\end{proof}
The previous theorem suggests the following definition
\begin{definition}[Laplace mean curvature lower bounds]
Let $(X,\de, \m)$ be an $\RCD(K,N)$ space for $K\in \R$, $N\in (1, \infty)$,  and let $\Omega\subset X$ be open. We say that $\partial \Omega$ has Laplace mean curvature bounded from by $H\in \R$ if 
\begin{align}\label{ineq:lapl}
{\bf \Delta}_{\Omega}(- \de_{\Omega^c})\geq - (N-1) \frac{s'_{\frac{K}{N-1}, \frac{H}{N-1}}}{s_{\frac{K}{N-1}, \frac{H}{N-1}}}\circ \de_{\Omega^c}\m|_{\Omega}.
\end{align}
\end{definition}

\begin{remark}
The direction "$\Leftarrow$" in Theorem \ref{th:laplacecomparison} holds any open $\Omega\subset X$ with $\Omega^c\neq \emptyset$ such that $\partial \Omega$ has Laplace mean curvature bounded from below. 
\end{remark}

%
%
%\begin{remark}\label{rem:ballcondition}
%As was pointed out to the author by the referee the property of having finite outer/inner curvature in the sense of the previous definition corresponds to an interior/exterior ball condition for $\Omega$ what is a condition on the full second fundamental form of the boundary $\partial \Omega$ in smooth context.
%\end{remark}
%\newpage
\section{Splitting}
\subsection{Measurable Splitting}\label{subsec:meas}
\begin{lemma}\label{lem:negsin}
Let $(X,d,\m)$ be essentially nonbranching and $MCP(K,N)$ for $K\in \R$ and $N\in (1,\infty)$. Let $\Omega\subset X$ be open  and set $u:= -d_{{\Omega^c}}$. 
Then $(\Omega^{c})^\circ\cap \T_{u,e}=\emptyset$, $\T_{u,e}\supset{\Omega}$ and $\mathfrak b_u\subset \partial \Omega$.
\end{lemma}
\begin{proof} First, we observe that for every $x\in \Omega$ there exists $y\in \Omega^c$ such that $-u(x)=d(x,y)$. Indeed, if $y_n\in {\Omega^c}$ is a minimal sequence, we have $y_n\in \overline{B_r(x)}$ for $r=-2u(x)$. Since $\overline{B_r(x)}$ is compact, there exists a converging subsequence and a limit point $y\in{\Omega^c}$. 

If $x\in (\Omega^c)^\circ$, then $u(x)=0$ and $(x,y)\in R_u$ only if
$$d(y,x)= - u(y).$$
Hence, if $x\neq y$, it follows that $y\in \Omega$ and there exists a geodesic $\gamma:[0,L]\rightarrow X$ between $x$ and $y$ such that $\gamma(t)\in \Omega$ for all $t\in (0,L)$. Consequently $x\in \partial \Omega$. This contradicts $x\in (\Omega^c)^\circ$. Therefore $x=y$ for all $y\in X$ such that $(x,y)\in R_u$. Hence $x\notin \T_{u,e}$ and $(\Omega^c)^\circ\cap \T_{u,e}=\emptyset$. 

Assume $x\in \Omega$. There exists $y\in {\Omega^c}$ and a geodesic $\gamma:[0,L]\rightarrow X$ such that  $L(\gamma)>0$ and 
$$d(x,y)=L(\gamma)=u(y)-u(x)=-u(x).$$
Therefore $x\in \T_{u,e}$ and $\Omega \subset \T_{u,e}$. This also implies $x\notin \mathfrak b$. Consequently $\mathfrak b\subset \partial \Omega$.
\end{proof}
\begin{corollary}\label{cor:constgrad}
One has $|\nabla u|=1$ {$\m$-a.e.} on $\Omega$.
\end{corollary}
\begin{proof}
Let $x\in \Omega$. As in the proof of the previous lemma there exist $y\in \partial \Omega$ and a geodesic $\gamma:[0,L]\rightarrow X$ such that $\gamma(0)=x$, $\gamma(L)=y$ and  $d(x,y)=L(\gamma)$. Moreover
\begin{align*}
1\geq |\nabla u|(x)=\lip u(x)=
\limsup_{y\rightarrow x} \frac{|u(x)-u(z))|}{d(x,z)}\geq \lim_{s\rightarrow 0} \frac{|u(x)-u(\gamma(s))|}{d(x,\gamma(s))}=1
\end{align*}
where we used the Sobolev-to-Lipschitz property in the first inequality. The first equality holds {$\m$-a.e.} and  is a fundamental result by Cheeger \cite{cheegerlipschitz}.
\end{proof}

Let $\gamma:[0,\infty)\rightarrow \overline\Omega$ be the geodesic ray such that $\gamma(0)\in \partial \Omega$, $\gamma((0,\infty))\subset \Omega$ and $\de_{{\Omega^c}}(\gamma(t))=t$. The Busemann function of $\gamma$ is defined as 
\begin{align*}
b(x)= \lim_{t\rightarrow \infty} d(x,\gamma(t))-t, \ x\in X.
\end{align*}
By triangle inequality the Busemann function is a welldefined and a $1$-Lipschitz map from $X$ to $\R$ that satisfies $b\in D({\bf\Delta}, \Omega)$.  {This is proved in \cite{giglistructure} and \cite{cav-mon-lapl-18}.}
The statement of the following Lemma appears in \cite{cav-mon-lapl-18}.
\begin{lemma}
$\mathfrak b_b=\emptyset$.
\end{lemma}
\begin{proof}
We pick $x\in X$ and consider the geodesic $\gamma^t:[0,L(\gamma^t)]\rightarrow X$ between $x$ and $\gamma(t)$. Clearly $L(\gamma^t)\rightarrow \infty$ for $t\rightarrow \infty$. Hence, $L(\gamma^t)>s>0$ for $s>0$ given and for $t>0$ sufficiently large. Since $\gamma^t$ is ageodesic we obtain that
\begin{align*}
s=d(\gamma^t(s),x)= d(\gamma^t(s),\gamma(t))-t -d(x,\gamma(t))+t.
\end{align*}
Let $z$ be an accumulation point of $\gamma^t(s)$, $t>0$.  Then taking $t\rightarrow \infty$ yields $d(x,z)=s=b(z)-b(x)$. 
Since $s>0$, it follows that $x\neq z$ and therefore $x\neq \mathfrak a_{b}$. 
\end{proof}

\begin{lemma}
Consider $X$ and $\Omega$ as in the previous lemma and assume $X$ is noncompact and ${\Omega^c}$ is compact. There exists a geodesic ray $\gamma: [0,\infty)\rightarrow X$ with $\gamma(0)\in \partial \Omega$, $\gamma((0,\infty))\subset \Omega$ and $\de_X(\gamma(0), \gamma(t))= \de_{\Omega^c}(\gamma(t))$. 
\end{lemma}
\begin{proof}
Since $X$ is noncompact and $\overline{\Omega^c}$ is compact, there exists a sequence $x_n\in X$ such that $d(x_n, \overline{\Omega^c})=:L_n\rightarrow \infty$. Let $\gamma_n: [0,L_n] \rightarrow X$ be the constant speed geodesic that connects $y_n\in \overline{\Omega^c}$ and $x_n$ such that $L(\gamma_n)=L_n$. It follows that $\mbox{Im}(\gamma_n)\subset \Omega´$. By  compactness of $\Omega^c$ there is a subsequence $(n_i)_{i\in \mathbb N}$such that $(\gamma_{n_i})$  uniformily converges on $[0,L_{n_0}]$ for any $n_0\in \N$ to a arclength parametrized geodesic ray $\gamma$ with $\gamma(0)\in \Omega^c$. Moreover $\mbox{Im}(\gamma)\subset \Omega$. Otherwise there is $t_0>0$ and a sequence $(t_n)_{n\in\N}$ such that $\Omega^c\ni \gamma_n(t_n)\rightarrow \gamma(t_0)\in \Omega^c$. 
Since $t_n=d(\gamma_n(t_n),\Omega^c)$, it follows $t_n\rightarrow 0$ and hence $t_0=0$ contradicting our assumption. Finally $\gamma:[0\infty)\rightarrow X$ also satisfies $\de_{X}(\gamma(0), \gamma(t))= \de_{\Omega^c}(\gamma(t))$. 
\end{proof}

\begin{proposition}
Let $(X,\de,\m)$ be $\RCD(K,N)$ and let $\Omega\subset X$ be connected with backward mean curvature bounded from below by $-\sqrt{(N-1)|K|}$. Let $u=\de_S|_{\Omega}$ and $\gamma:(0, \infty)\rightarrow \Omega$ a geodesic ray, such that $\lim_{t\downarrow 0}\gamma(t)= x\in \partial \Omega$ and $t=\de_X(x,\gamma(t))= \de_{\Omega^c}(\gamma(t))$. Let $b$ be the associated Busemann function as before. Assume $\Omega$ is connected. 
Then $b|_{\Omega}=-u$ and 
\begin{align}\label{idid}
{\bf \Delta}_{\Omega}(b|_{\Omega})=(N-1)\sqrt{\frac{|K|}{N-1}}\m|_{\Omega^{}}\ \ \&\ \ {\bf \Delta}_{\Omega}u=- (N-1)\sqrt{\frac{|K|}{N-1}}\m|_{\Omega^{}}.
%(N-1) \frac{s'_{\frac{K}{N-1}, \frac{H}{N-1}}(b|_{\Omega^\circ})}{s_{\frac{K}{N-1}, \frac{H}{N-1}}(b|_{\Omega^\circ})}\m|_{\Omega^{\circ}}.
\end{align}
In particular $\mathfrak a_b= \mathfrak b_u=\emptyset$.
\end{proposition}
\begin{proof}
The $\CD(K,N)$ condition yields 
\begin{align*}
{\bf \Delta}_{\Omega } (b|_{\Omega})\leq (N-1)\sqrt{\frac{|K|}{N-1}}\m|_{\Omega^{}}.
\end{align*}
%Following the proof of the Laplace comparison estimate of the Busemann function on $\CD(0,N)$ spaces in  \cite{giglistructure} we also get  that $${\bf \Delta}_{\Omega^{\circ}} (b|_{\Omega^{\circ}})\geq -  (N-1) \frac{s'_{\frac{K}{N-1}, \frac{H}{N-1}}(b|_{\Omega^\circ})}{s_{\frac{K}{N-1}, \frac{H}{N-1}}(b|_{\Omega^\circ})}\m|_{\Omega^{\circ}}.$$ 
Hence with the Laplace estimate for $u=  d_S|_{\Omega^{}}= -d_{{\Omega^c}}|_{\Omega}$ we obtain 
\begin{align*}
{\bf \Delta}_{\Omega} (b-u) = {\bf \Delta}_{\Omega} b - {\bf \Delta}_{\Omega^{}} u\leq (N-1)\sqrt{\frac{|K|}{N-1}}\m|_{\Omega^{}}- (N-1)\sqrt{\frac{|K|}{N-1}}\m|_{\Omega^{}} =0.
\end{align*} 
Pick $y\in \partial \Omega$ such that $\de(x,y) = \de_{{\Omega^c}}(x)$. Then
\begin{align*}
\de(x,\gamma(t))-t + \de_{{\Omega^c}}(x) \geq \de(y,\gamma(t))\geq \de_{{\Omega^c}}(\gamma(t))-t=0
\end{align*}
and it follows $b(x)- u(x)\geq 0$ for $x\in \Omega^{}$ where we used  $\de_{\Omega^c}(\gamma(t))=\inf_{z\in \Omega^c} \de_X(x,\gamma(t))= t= \de_X(\gamma(0), \gamma(t))$ in the last equality. 
Moreover, equality  holds if $x=\gamma(s)$ for some $s>0$.

%Moreover, $b\circ \gamma(t)=-u\circ \gamma(t)$ for $t>0$. 
By the maximum principle for $\RCD$ spaces  \cite{gigli_rigoni, giglimondino} it follows that $b=u$ on $\Omega$ and $${\bf \Delta}_{\Omega} (b|_\Omega)={\bf \Delta}_{\Omega} u$$
which by linearity of the Laplacian yields the identity \eqref{idid}. 
\end{proof}
\begin{corollary}
% 
% Let $(X,d,\m)$ be $\RCD(0,N)$ for $N\in (1,\infty)$. We consider two situations simultaneously. In the first situation we asume that $X$ is noncompact and we consider a compact subset $\Omega\subset X$. In the section situation we consider two disjoint compact subsets $\Omega_0, \Omega_1\subset X$. 
Consider $(X,\de,\m)$, $b$, $\Omega$ and $u$ as before and  the $1D$ localisation $(X_\gamma)_{\gamma\in Q}$ w.r.t. $u=-b|_{\Omega^{\circ}}$ on $\Omega^{\circ}$ where $\gamma: [0,\infty) \rightarrow \Omega$ $\forall \gamma\in Q$. and the corresponding  disintegration of $\m|_{\Omega^{\circ}}$ into  measures  $(\m_\gamma)_{\gamma\in Q}$.
Then $\m(\Omega^\circ\backslash \T_{u}^\dagger)=0$ and $$\m_\gamma= h_\gamma(0)s_{\frac{K}{N-1}, -\sqrt{\frac{|K|}{N-1}}}(r)^{N-1}\mathcal H^1|_{[0,\infty)}(r).$$ In particular 
\begin{align*}
\frac{\m(B_R(\Omega^c)\cap \Omega^{})}{\m(B_r(\Omega^c)\cap\Omega^{})} = \frac{\int_0^R s_{\frac{K}{N-1}, -\sqrt{\frac{|K|}{N-1}}}(t)^{N-1} dt}{\int_0^r s_{\frac{K}{N-1}, -\sqrt{\frac{|K|}{N-1}}}(t)^{N-1} dt}.
\end{align*}
\end{corollary}
%\begin{proof}
%The first claim follows from Lemma \ref{somelemma}.  Let $\gamma\in Q^{\dagger}$. By previous Theorem \ref{th:1dlocalisation} the function $h_\gamma: [0,\infty)\rightarrow [0,\infty)$ is constant. Hence, $\gamma\in Q^{\dagger}\mapsto h_\gamma=h_\gamma(0)=:c(\gamma)$ is a measurable function and $\m_\gamma=c(\gamma)\mathcal H^1$.
%\end{proof}

Let $\Omega\subset X$ be connected, not empty and given by $\Omega= \bigcap_{\alpha=1}^m \Omega_\alpha$ for $\Omega^c_\alpha\cap \Omega^c_\beta=\emptyset $ and $\de(\Omega_\alpha, \Omega_\beta)= D_{\alpha, \beta}>0$ for $\alpha\neq \beta$ and $m\in \mathbb N$.  We set $S_\alpha= \partial \Omega_\alpha$ and $u_\alpha = -d_{S_\alpha}|_{\Omega^{}}$, $\alpha=1, \dots, m$.

\begin{lemma}\label{lem:harmonic}
Let $(X,\de,\m)$ be an $\RCD(K,N)$ space, and  let $\Omega$, $\Omega_\alpha$, $\alpha=1,\dots, m$ as before. Assume $\partial \Omega_\alpha$, $\alpha\neq 2$,  has backward mean curvature bounded from below by $\sqrt{|K|}$ and $\partial \Omega_2$ has backward mean curvature bounded from by below by $-\sqrt{|K|}$. Moreover, assume that $\partial \Omega_2$ is compact.
Then $m=2$ and  $-u_1= d(\overline\Omega_2,\overline\Omega_1)+u_2$ and $${\bf \Delta}_{\Omega^{\circ}} u_1= - (N-1) \sqrt{\frac{|K|}{N-1}}\m|_{\Omega^{\circ}} \ \ \ \& \ \ \  {\bf \Delta}_{\Omega^{\circ}} u_2=  (N-1) \sqrt{\frac{|K|}{N-1}}\m|_{\Omega^{\circ}} .$$
\end{lemma}
\begin{proof}
Consider $\Omega_1$ and $\Omega_2$. Since $\partial \Omega_2$ is compact, there are points $y_i\in \partial \Omega_i$, $i=1,2$, such that $\de(\Omega_1^c, \Omega_2^c)= \de(y_1, y_2)=D_{1,2}$. Moreover, the geodesic $\gamma:[0, D_{1,2}]\rightarrow \Omega$ from $y_2$ to $y_1$ satisfies 
\begin{align}\label{estimate}
u_1(\gamma(t)) + u_2(\gamma(t))=- D_{1,2}\geq  u_1(x)+u_2(x)\ \ \forall x\in \Omega.
\end{align}
By Theorem \ref{th:laplacecomparison}
\begin{align*}
{\bf \Delta}_{\Omega_1^{}} u_1\leq   -(N-1) \sqrt{\frac{|K|}{N-1}}\m|_{\Omega_1^{}} \ \ \ \&\ \ \ \  {\bf \Delta}_{\Omega_2^{}} u_2\leq  (N-1) \sqrt{\frac{|K|}{N-1}}\m|_{\Omega_2^{}}.
\end{align*}
Hence $\Delta_{\Omega_1^{}\cap \Omega_2^{}} (u_1+ u_2)\leq 0$. Since we have \eqref{estimate} by the maximum prinicple it follows 
\begin{align*}
{\bf \Delta}_{\Omega_1^{}\cap \Omega_2^{}}(u_1+ u_2)=0 \ \ \& \ \ u_1=-u_2- D_{1,2}\mbox{ on }\Omega_1^{}\cap \Omega_2^{}.
\end{align*}

Now assume that $l\geq 3$. Set $\de({\Omega_i^c}, {\Omega_j^c})=D_{i,j}$. 
Similarly as before one deduces that 
\begin{align*}
u_3+ u_2= - D_{3,2} \mbox{ on } \Omega_2\cap \Omega_3.
\end{align*}
Together with the equation for $u_1$ and $u_2$ it follows 
\begin{align*}
u_1-u_3= D_{2,3}- D_{1,2} \mbox{ on } \Omega_1\cap \Omega_2\cap \Omega_3.
\end{align*}
%If $D_{2,3}= D_{1,2}$, then $u_1=u_3$ on $\Omega_1\cap \Omega_2\cap \Omega_3$.  Also 
Note that  $\partial \Omega_1, \partial \Omega_3\subset \overline{\Omega_1\cap \Omega_2\cap \Omega_3}$.
%and since $x\in\partial \Omega_i$ $\Leftrightarrow$ $u_i(x)=0$, it follows that $\partial \Omega_1=\partial \Omega_3$ which contradicts $\overline{\Omega_1^c}\cap \overline{\Omega_3^c}$.
%
Assume w.l.o.g.  that  $D_{1,2}\geq  D_{2,3}$.  It holds 
\begin{align*}
\mbox{$x\in \partial \Omega_1$ $\Leftrightarrow$ $u_1(x)=0$ $\Leftrightarrow $ $u_3(x)= D_{1,2}- D_{2,3}\geq 0$ $\Leftrightarrow$ $x\in {\Omega_3^c}$.}
\end{align*}  Hence $x\in {\Omega_3^c}\cap {\Omega_1^c}$.  This is a contradiction.
\end{proof}
\begin{corollary}
\begin{align*}
\frac{\m(B_R(\Omega_1^c)\cap \Omega)}{\m(B_r(\Omega^c_1)\cap \Omega)}=  \frac{\int_0^R s_{\frac{K}{N-1}, -\sqrt{\frac{|K|}{N-1}}}(t)^{N-1} dt}{\int_0^r s_{\frac{K}{N-1}, -\sqrt{\frac{|K|}{N-1}}}(t)^{N-1} dt}.
\end{align*}

\end{corollary}
\subsection{Isometric splitting}\label{subsec:isometric}
%By rescaling in the following we can assusme $K= \delta (N-1)$ for $\delta=0,1$.
Recall that $f\in W^{1,2}(\Omega^{\circ})$ if $\phi\cdot f\in W^{1,2}(X)$ for every Lipschitz function with support in $\Omega^{}$. Moreover, we say  $u\in H^{2,2}_{loc}(\Omega)$ if $\psi\cdot  u\in H^{2,2}(X)$ for every $\psi\in \mathbb D_\infty$ with support in $\Omega^{}$.  Thanks to locality of $\mbox{Hess} f$ for $f\in H^{2.2}(X)$ the Hessian $\mbox{Hess}(u)$ for $u\in H^{2,2}_{loc}(\Omega)$ is well-defined.

The following theorem is Corollary 4.16 in \cite{kkl}.
\begin{theorem}\label{th:hessian}
Let $X$ be $\RCD(0, N)$ and $\Omega\subset X$ be open.  Let $u: \Omega\rightarrow \R$ such that $|\nabla u|=1$ and ${\bf \Delta}_{\Omega} u= 0$.  Then $u\in H_{loc}^{2,2}(\Omega^{})$ and 
\begin{align*}
\Hess(u)(\nabla f, \nabla f) =0 \ \m\mbox{-a.e.}\mbox{ on }\  \Omega^{} \mbox{ and } f\in \mathbb{D}_\infty.
\end{align*}
\end{theorem}
\begin{remark}\label{rem:example}
Given an open subset $\Omega\subset X$ of an $\RCD$ space $X$ we define  $(\tilde \Omega, \tilde{\de}_\Omega)$ as the completion of $\Omega$ equipped with the  intrinsic distance induced by $\de_X$. We can identify $\Omega$ as a subset of   $\tilde \Omega$, but the topology of $(\tilde \Omega,\tilde \de_{\Omega})$ can differ from the topology of $\overline\Omega\subset X$.  An easy example for this scenario is $X=\mathbb S^1$ and $\Omega= \mathbb S^1\backslash \{p\}$ with $p\in \mathbb S^1$. The completion of $\Omega$ equipped with the intrinsic distance is an interval. But $\overline \Omega=\mathbb S^1$.  

Setting $\m|_{\Omega}=\m_{\Omega}$ the triple $(\tilde \Omega, \tilde \de_{\Omega}, \m_{\Omega})$ is a metric measure space.
\end{remark}
% for $\Omega$ equipped with its induced intrinsic distance. 

A corollary of Theorem \ref{th:hessian} is the following splitting result. 
\begin{theorem}
Let $(X,\de,\m)$, $\Omega\subset X$ and $u$ be as in previous theorem. Assume that $\Omega=u^{-1}((0,D))$ for $D>0$.
Then, there exists an $\RCD(0,N-1)$ space $(Y, \de_Y, \m_Y)$ such that $(\tilde \Omega, \tilde \de_{\Omega}, \m_{\Omega})$ is isomorphic to $[0, D]\otimes Y$.
\end{theorem}
\begin{proof} 
The proof of the corollary is  exactly the content of section 5 and section 6 in \cite{kkl} that result in the proof of Theorem 6.10 in \cite{kkl} that corresponds to our statement. 
\end{proof}
\begin{remark}
For the proof of the main theorem in \cite{kkl} the authors show that the induced intrinsic metric of $\Omega= f^{-1}((-\min f, \max f))$ splits off an interval where $f=\cos^{-1}\circ u$ with an eigenfunction $u$ on a compact $\RCD(0,N)$ space $X$. 
% It follows ${\bf \Delta}_{\Omega} f|_\Omega=0$ and $|\nabla f|=1$ on $X$ that also implies $\mbox{Hess} f=0$ on $\Omega$.
\end{remark}
As consequence of the previous theorem one obtains the following isomorphic splitting statement that generalizes a corresponding theorem in smooth context by Kasue \cite{Kasue83} and Croke-Kleiner \cite{crokekleiner}.

\begin{theorem}
Let $(X,\de,\m)$ be an $\RCD(0,N)$ space, and  let $\Omega$, $\Omega_\alpha$, $\alpha=1,\dots, m$ as before. Assume $\partial \Omega_\alpha$  has backward mean curvature bounded from below by $0$ for every $\alpha=1, \dots, m$. Moreover, assume that $\partial \Omega_2$ is compact.
Then, there exists an $\RCD(0,N-1)$ space $(Y, \de_Y, \m_Y)$ such that $(\tilde \Omega, \tilde \de_{\Omega}, \m_{\Omega})$ is isomorphic to $[0, D_{1,2}]\otimes Y$.
\end{theorem}

\begin{proof}[Proof of Corollary \ref{cor1}] Since mean curvature bounded from below by $\delta>0$ implies nonnegative mean curvature,  we can apply Theorem \ref{main1}. It follows that ${\bf \Delta}_{\Omega}(-\de_{\Omega^c})=0$. But $${\bf \Delta}_{\Omega}(-\de_{\Omega^c})\geq (N-1)\frac{\frac{\delta}{N-1}}{1-\frac{\delta}{N-1}\de_{\Omega^c}}>0 \mbox{ on } B_1(\Omega^c)\cap \Omega$$
by the assumed mean curvature bound. This is a contradiction.
\end{proof}
\begin{proof}[Proof of Corollary \ref{cor2}] Recall that for $K\leq K'$ the condition  $\RCD(K', N)$ implies $\RCD(K,N)$.  Assume $\Omega^c_1$ and $\Omega^c_2$ are disjoint and set $\Omega=\Omega_1\cap \Omega_2$. Then by Theorem \ref{main1}  $(\tilde \Omega, \tilde\de_{\Omega}, \m|_{\Omega})$ is isomorphic to $[0,D]\times Y$ for some $\RCD(0,N-1)$ space $Y$. But the product structure contradicts  the assumption that $X$ was $\RCD(\delta, N)$ for $\delta>0$.
\end{proof} 
Similarly one can show the following splitting theorem. 
\begin{theorem}
Let $(X,\de,\m)$ be $\RCD(0,N)$ and let $\Omega\subset X$ have backward mean curvature bounded from below by $0$. Assume $\Omega^{\circ}$ is connected and there exists a geodesic ray $\gamma: (0, \infty) \rightarrow \Omega$ with $\lim_{r\downarrow 0} \gamma(r)=x_0\in \partial \Omega$ and $\de_{X}(\gamma(0), \gamma(t))=\de_{\Omega^c}(\gamma(t))$.
Then, there exists an $\RCD(0, N-1)$ space $(Y, \de_Y, \m_Y)$ such that $(\tilde \Omega, \tilde \de_\Omega, \m_{\Omega})$ is isomorphic to $[0,\infty) \otimes Y$.
\end{theorem}
Again, the proof is verbatim the same as for  \cite[Theorem 610]{kkl}. Noncompactness only requires minor modifications since the arguments are all of local nature.

\section{Almost rigidity}

\subsection{Gromov-Hausdorff convergence and the uniform distance} {In this and the following sections we will study the stability and almost rigidity properties of lower mean curvature bounds. For simplicity, we  will assume that  all the involved $\RCD$ spaces are compact. An extension of the following concepts for non-compact $\RCD$ spaces and pointed Gromov-Hausdorff convergence is omitted but straightforward. }

Compact metric spaces $(X_i, \de_i)$ converge in Gromov-Hausdorff sense to a compact metric  spaces $(X, \de)$ if there exist a compact metric space $(Z,\de_Z)$ and distance preserving maps $\iota_i, \iota: X_i, X \rightarrow Z$ such that $\iota_i(X_i)$ converges in Hausdorff sense to $\iota(X)$ in $Z$. The Gromov-Hausdorff distance $\de_{GH}(X_i, X)$ is  defined as the infimum of Hausdorff distances between $\iota_i(X_i)$ and $\iota(X)$ w.r.t. to all distance preserving maps $\iota_i, \iota$ and metric spaces $Z$.
Equivalently,  $(X_i, \de_i) \overset{\scriptscriptstyle{GH}}{\rightarrow} (X,\de)$ if there exists a sequence of $\epsilon_i$-isometries $\psi_i: X_i\rightarrow X$ such that $\epsilon_i\rightarrow 0$.  Existence of an $\epsilon$-isometry $\psi: X\rightarrow Y$ between compact metric spaces $X$ and $Y$ yields that  the Gromov-Hausdorff distance satisfies $\de_{GH}(X,Y)\leq 2\epsilon$.

Given a sequence of $\delta_i$-isometries $\psi_i: X_i \rightarrow X$ with $\delta_i\rightarrow 0$ a sequence of functions $f_i: X_i \rightarrow \R^m$ converges uniformily to a function $f:X\rightarrow \R^m$ if for every $\epsilon>0$ there exists $i_\epsilon\in \mathbb N$ such that $\left\|f_i(z_i)- f(z)\right\|_{\R^m}\leq \epsilon$ for points $z_i\in X_i$ and $z\in X$ with  $\de_{Z}(\psi_i(z_i), z)\leq \delta_i$ and $i\geq i_\epsilon$.

The next proposition is Gromov's Arzela-Ascoli theorem for functions on a Gromov-Hausdorff converging sequence (for instance see \cite{saa}).

\begin{proposition} Let $(X_i, \de_i)$ be compact metric spaces that converge in GH sense to a compact metric space $(X,\de)$, and let $f_i: X_i \rightarrow \R^m$ be functions that are $L$-Lipschitz and uniformly bounded. Then there exists a subsequence of $f_i$ that converges uniformly to an $L$-Lipschitz function $f: X\rightarrow \R^m$.
\end{proposition}
These considerations motivate the following definitions. 

Let $X$ and $Y$ be compact metric spaces such that $\de_{GH}(X, Y)<r$. Then it is easy to see that there exist  $2r$-isometries $\psi: X\rightarrow Y$ and $\phi: Y\rightarrow X$. 
\begin{definition}[Uniform distance]\label{def:ud}For functions $f: X\rightarrow \R^m$ and $g: Y\rightarrow \mathbb R^m$ we define
\begin{align*} 
\sup\big\{ \left\|f(x)- g(y)\right\|_{\R^m}: x\in X, y\in Y\ \mbox{s.t.} \ \de_X(\psi(x), y))\leq 2r  \big\}=: \mathcal S_{\psi}(f,g).
\end{align*}
The uniform distance between $f$ ang $g$ is then defined via
\begin{align*}
\inf_{(\psi, \phi)} \max \left\{ \mathcal S_\psi(f,g), \mathcal S_\phi(f,g)\right\}=: \de^{\star}(f,g)
\end{align*}
where the infimum is taken w.r.t.  any pair $(\psi, \phi)$ such that $\psi: X \rightarrow Y$ and  $\phi: Y\rightarrow X$ are $2r$-isometries for $r>\de_{GH}(X,Y)$.
\end{definition}
By definition we have $\de^\star(f,g)= \de^\star(g,f)$,  and $\de_{GH}(X,Y)+ \de^\star(f,g)=0$ if and only if $X\simeq Y$ and {$f=g$ pointwise as functions on $X\simeq Y$}. Moreover, for compact metric spaces $X, Y$ and $Z$, and continuous functions $f: X\rightarrow \R^m$, $g: Y\rightarrow \R^m$ and $h: Z\rightarrow \R^m$ we have
\begin{align*}
\de^{\star}(f,h) \leq  \de^\star(f,g) +  \de^\star(g,h).
\end{align*}
\begin{lemma}
 Consider compact metric spaces $(X_i, \de_i)$ for $i\in \mathbb N$ and $(X,\de)$ such that $\de_{GH}(X_i, X)<r_i \rightarrow 0$. Then  $f_i: X_i\rightarrow \R^m$ converges uniformly to $f:X\rightarrow \R$ if and only if $\de^{\star}(f_i, f)\rightarrow 0$.
\end{lemma}
\begin{proof} 
Let $\epsilon>0$, then we can pick $i_\epsilon>0$ such that $\de^{\star}(f_i, f)\leq \epsilon$ for $i\geq i_\epsilon$. In particular, there exists a sequence of $2r_i$-isometries $\psi_i:X_i \rightarrow X$ such that 
\begin{align*}
\left\| f_i(x_i)- f(x)\right\|_{\R^m} \leq \epsilon \ \ \forall x_i\in X_i, x\in X \mbox{ with } \de(\psi_i(x_i), x)\leq 2r_i \ \forall i\geq i_\epsilon.
\end{align*}
Hence, $f_i$ converges uniformly to $f$. 
On the other hand, the definition of uniform convergence implies $\de^{\star}(f_i, f)\rightarrow 0$.
\end{proof}
%
%If $f=(f^1, \dots, f^l)$ and $g=(g^1, \dots, g^l)$ are $\R^l$-valued functions for $l\in \mathbb R$, we define 
%\begin{align*}
%\de_{AA}(f,g):= \sum_{i=1}^l \de_{AA}(f^i, g^i).
%\end{align*}
%
%Given two  open subsets $\Omega_1\subset X$ and $\Omega_2\subset Y$ with $\partial \Omega_i=S_i$  and the corresponding signed distance functions $\de_{S_i}$, we define
%\begin{align}
%\de^\star (S_1, S_2)= \de_{AA} (\de_{S_1}, \de_{S_2}).
%\end{align}

Let $i=1,2$. Given families of open sets $\Omega_{i,\alpha}\subset X_i$, $\alpha=1, \dots, m$ such that $\Omega_{i,\alpha}^c$ is connected for all $\alpha$ and $\de_{X_i}(\Omega_{i, \alpha}^c, \Omega_{i, \beta}^c)=\inf_{x\in \Omega_{i,\alpha}^c, y\in \Omega_{i, \beta}^c}\de(x,y)>0$ for $\alpha \neq \beta$,  we consider $\Omega_i = \bigcap_{l=1}^m \Omega_{i,l}$ and  $f_i=(\de_{\Omega^c_{i,1}}, \dots, \de_{\Omega^c_{i,m}}): X_i \rightarrow \R^m$. Then we define
\begin{align}
\mathcal D(\Omega_1, \Omega_2) := \de^{\star}(f_1, f_2).
\end{align}

A sequence of compact metric measure spaces $(X_i, \de_i, \m_i)$ converges in measured Gromov-Hausdorff sense to a compact metric measure spaces $(X,\de,\m)$   if  $(X_i, \de_i) \overset{\scriptscriptstyle GH}{\rightarrow} (X,\de)$ and $\m_i$ converges to $\m$ in duality with $C_b(Z)$ where $(Z, \de_Z)$ is a metric space where GH convergence is realized. A distance that metrizes measured GH convergence is given for instance by Sturm's \textit{tranportation distance} $\mathbb D$ \cite{stugeo1}. Actually $\mathbb D$ is a distance on the set of isomorphism classes $[X]$ of metric measure spaces $X$ with finite measure $\m_X$. But after normalisaton of $\m_X$, that is replacing $\m_X$ with ${\m_X(X)}^{-1} \m_X= \bar \m_X$,  we can see $\mathbb D$ is a distance on the family of normalized metric measure spaces. Then $\mathbb D$ is estimated by the $L^2$-Wasserstein distance $\de_W^Z(\bar \m_i, \bar \m)$ in $Z$ between the normalisations of $\m_i$ and $\m$.
%\subsection{convergence} Recall that a sequence of normalized metric measure spaces $(X_i, \de_i, \m_i)$ converges in measured Gromov-Hausdorff sense to $(X,\de, \m)$ if one can embed $X_i$ and $X$ into a common metric space $(Z, \de_Z)$ where $\m_i$ converges to $\m$ in duality to $C_{bs}(X)$. 

A sequence of functions $f_i\in L^2(\m_i)$ converges in $L^2$-weak sense to $f\in L^2(\m)$ if $f_i \m_i \rightarrow f\m$ in duality with $C_{b}(Z)$ and $\sup_{i\in \mathbb N} \left\| f_i \right\|_{L^2(\m_i)}<\infty$. If $$\lim_{i\rightarrow \infty} \left\| f_i\right\|_{L^2(\m_i)} = \left\| f\right\|_{L^2(\m)}$$ holds, then one says the sequence $f_i$ converges $L^2$-strongly to $f$. 

\begin{lemma}  Let $f_i, g_i\in L^2(\m_i)$ such that $f_i\rightarrow f\in L^2(\m)$, $g_i\rightarrow g\in L^2(\m)$ $L^2$-strongly. Then
\begin{enumerate}
\item $f_i+g_i$ converges $L^2$-strongly to $f+g$. 
\smallskip
\item $\int |f_i g_i| d\m_i \rightarrow \int |g f| d\m$.
\end{enumerate}
\end{lemma}
A sequence of functions $f_i\in W^{1,2}(X_i)$ converges $H^{1,2}$-weakly to $f\in W^{1,2}(X)$ if $(f_i)$ converges $L^2$-weakly to $f$ and $\int_X\left| \nabla f_i\right|^2 d\m_i<\infty$. The sequence $(f_i)$ converges $H^{1,2}$-strongly if $(f_i)$ converges $L^2$-strongly and $$\lim_{i\rightarrow \infty} \int_X|\nabla f_i|^2 = \int_X|\nabla f|^2 d\m.$$
%
%
%A sequence of functions $f_i\in W^{1,2}(B_r(x_i))$ converges $W^{1,2}$-weakly to $f\in W^{1,2}(B_r(x))$ if $(f_i)$ converges $L^2$-weakly to $f$ and $\int_X\left| \nabla f_i\right|^2 d\m_i<\infty$. The sequence $(f_i)$ converges $W^{1,2}$-strongly if $(f_i)$ converges $L^2$-strongly and $\lim_{i\rightarrow \infty} \int_{B_r(x_i)}|\nabla f_i|^2 = \int_{B_r(x_i)}|\nabla f|^2 d\m$.
%\begin{lemma}
%For any $f\in \lip_c(B_r(x))$ there exist $f_i\in \lip_c(B_r(x_i))$ such that 
%\begin{align}
%\sup_i\left\| \nabla f_i\right\| <\infty 
%\end{align}
%and $(f_i)$ converges $W^{1,2}$-strongly to $f$.
%\end{lemma}
\subsection{Stability and almost rigidity results}
\begin{lemma}\label{lem:AA} Let $K\in \mathbb R$ and $N\in (1,\infty)$. 
Let  $(X_i, \de_i, \m_i)_{i\in \mathbb N}$, be a sequence of $\RCD(K,N)$ spaces that converges in measured Gromov-Hausdorff sense to a compact $\RCD(K,N)$ space $(X, \de, \m)$, and let $\Omega_i\subset X_i$ be open sets.  Then, $-\de_{\Omega_i^c}: X_i\rightarrow \mathbb R$ subconverges in Arzela-Ascoli sense to a $1$-Lipschitz function $u:X\rightarrow \R$ such that $|\nabla u|=1$ {$\m$-a.e.} on $\Omega= u^{-1}((-\infty,0))$ and $\Omega^c=u^{-1}(\{0\})\neq \emptyset$.
Moreover $u=-\de_{\Omega^c}$ if $\Omega\neq \emptyset$. Otherwise $u\equiv 0$.
\end{lemma}
\begin{proof} 
The existence of a $1$-Lipschitz function $u: X\rightarrow \R$ that arises as  the limit of  a subsequence of $\de_{\Omega^c_i}$ is guaranteed by Gromov's Arzela-Ascoli theorem.  

We embed $(X_i, \de_i)$ and $(X,\de)$ into a metric space $(Z,\de_Z)$ where measured Gromov-Hausdorff convergence is realized.
Assume $\Omega=u^{-1}((-\infty,0))\neq \emptyset$. Then we pick $x\in \Omega$ and a sequence of points $x_i\in X_i$ such that $x_i\in \Omega_i$ and $x_i\rightarrow x$ in $Z$. There exists a sequence of geodesics $\gamma_i: [-L_i,0] \rightarrow \Omega_i$ that are arclength parametrized such that $\gamma_i(-L_i)=x_i$, $u(\gamma_i(0))=0$ and $L_i=u(x_i)$. After extracting another subsequence $(\gamma_i)_{i\in \mathbb N}$ converges uniformily to a geodesic $\gamma: [-L,0]  \rightarrow X$ in $Z$ such that $\gamma(-L)=x$, $L=u(x)>0$ and $\gamma((-L,0))\subset \Omega$. It holds
\begin{align*}
u(\gamma(-L))- u(\gamma(0))= \de_{X}(x, \gamma(0)).
\end{align*}
Hence $\gamma$ is a transport geodesic of $u$ and $x$ is contained in the transport set $\mathcal{T}_u$. Hence $\Omega \subset \mathcal{T}_u$ and $\lip u =|\nabla u|=1$ on $\mathcal{T}_u$.

If we assume there exits $y\in \Omega^c$ such that $\de(x,y)<L$, then there exist $y_i\in \Omega_i^c$ such that $y_i\rightarrow y$ and $\de_i(x_i,y_i)\rightarrow \de(x,y)$. This would contradict the choice of $\gamma_i$ before. Hence $-u(x)=L=\de_{\Omega^c}(x)$. Hence $-u=\de_{\Omega^c}$.
\end{proof}
\begin{definition}[uniform domain] \label{def:undo} Let $X$ be a geodesic metric space.
An open subset $\Omega\subset X$ is called $(c,C)$-uniform if for any two points $x,y\in \Omega$ there exists a rectifiable curve $\gamma: [0,1] \rightarrow \Omega$  with $\gamma(0)=x$ and $\gamma(1)=y$ that  satisfies 
\begin{enumerate}\item $\de_{\Omega^c}(\gamma(t))\geq c \min\{ \de_X(x, \gamma(t)), \de_X(\gamma(t), y)\}\ \forall t\in [0,1]$, 
\item $
\mbox{length}(\gamma)\leq C \de_X(x,y)$.
\end{enumerate}
In particular, a $(c,C)$-uniform domain is connected. 
\end{definition}
\begin{lemma}\label{lem:uniformdomain}
Consider $X_i$, $X$, $\Omega_i$ and $\Omega$ as in  Lemma \ref{lem:AA}. If $\Omega_i$ is $(c,C)$-uniform for all $i\in \mathbb N$, then $\Omega$ is $(c,C)$-uniform. If $\Omega_i\neq \emptyset$ for all $i$, then $\Omega\neq \emptyset$.
\end{lemma}
\begin{proof}
Pick two points $x,y\in \Omega$ and $x_i, y_i\in \Omega_i$ such that $x_i\rightarrow x$ and $y_i\rightarrow x$ after embedding $X_i, X$ into a common metric space $Z$. 

Since $\Omega_i$ is $(c,C)$-uniform, there exists a sequence of rectifiable curves $\gamma_i: [0,1] \rightarrow \Omega_i$  that connects $x_i$ and $y_i$ and such that $\mbox{length}(\gamma_i)\leq C\de(x_i, y_i)$. We apply the Arzela-Ascoli theorem  to extract a subsequence that converges uniformily in $Z$ to a curve $\gamma: [0, 1]\rightarrow X$. Lower semi-continuity of the length implies  that $\gamma$ is rectifiable and
\begin{align*}
\mbox{length}(\gamma)\leq C\de(x,y)
\end{align*}
Moreover, uniform convergence of $\de_{\Omega_i^c}$ and  convergence of $\gamma_i$ implies $$\de_{\Omega^c}(\gamma(t))\geq c\min\{\de_X(x,\gamma(t)), \de_X(\gamma(t),y)\}.$$
Hence $\Omega$ is $(c,C)$-uniform.

The second claim is clear.
\end{proof}
\begin{theorem}\label{th:meancurvaturestability}
Consider $X_i$, $X$, $\Omega_i$, $\Omega$ as in Lemma \ref{lem:AA} such that $\Omega\neq \emptyset$. We set $u_i:= \de_{\Omega_i^c}|_{\Omega_i}$ and $u:=\de_{\Omega^c}|_{\Omega}$.  Assume $u_i$ satisfies 
\begin{align*}
{\bf \Delta}_{\Omega_i^{}} u_i\leq  (N-1) \frac{s'_{\frac{K}{N-1}, \frac{H_i}{N-1}}(u_i)}{s_{\frac{K}{N-1}, \frac{H_i}{N-1}}(u_i)}\m|_{\Omega_i^{}}
\end{align*} where $H_i\in \mathbb R$ with $H_i\rightarrow H$.
Then $u$ satisfies 
\begin{align*}
{\bf \Delta}_{\Omega^{}} u\leq (N-1) \frac{s'_{\frac{K}{N-1}, \frac{H}{N-1}}(u)}{s_{\frac{K}{N-1}, \frac{H}{N-1}}(u)}\m|_{\Omega^{}}.
\end{align*}
\end{theorem}
\begin{proof} 
%We first show that uniform convergence of $\de_{\Omega_i^c}$ to $\de_{\Omega^c}$ implies $L^2$-strong convergence. 
By measured Gromov-Hausdorff convergence there
exists a compact metric space $(Z,\de_Z)$, distance preserving maps $\iota_i, \iota: X_i, X \rightarrow Z$ and couplings $\pi_i$ between $\m_i$ and $\m$ such that $\de_Z(x,y)\leq \delta$ for $\pi_i$-almost every $(x, z)\in X_i\times X$ if $i\geq i_\delta$. Let $\phi\in C_b(Z)$ and define $g_i= \phi \cdot \de_{\Omega^c_i}$. Then $g_i$ converges uniformily to $g=\phi \cdot\de_{\Omega^c}$, and we can choose $i_\delta\in \mathbb N$ such that $|g_i(x_i) -g(x)|<\epsilon$ whenever $|x_i-x|\leq \delta$ and $i\geq i_\delta$.  Indeed, we observe
\begin{align*}
&|\phi(x) \cdot \de_{\Omega^c}(x) - \phi(y)\cdot \de_{\Omega_i^c}(y)|\\
&\leq |\phi(x)| |\de_{\Omega^c}(x) - \de_{\Omega^c_i}(y)|+ |\phi(x) -\phi(y)| \de_{\Omega_i}(y)\\
&\leq \sup_{z\in Z} |\phi(z)|  \epsilon + \epsilon \cdot \diam_{X_i}
\end{align*}
whenever $i\geq i_\delta$ is sufficiently large and $\de_Z(x,y)\leq \delta$.

It follows that 
\begin{align*} 
\left|\int g_i d\m_i - \int g d\m\right| = \int \left|g_i-g\right|d\pi_i \leq  \epsilon \ \mbox{ for }\  i\geq i_{\delta}.
\end{align*}
It follows that $\de_{\Omega_i^c}\m_i\rightarrow \de_{\Omega^c}\m$ in duality with $C_b(Z)$. 
Moreover 
\begin{align*} 
\left|\int \de_{\Omega^c_i}^2 d\m_i - \int \de_{\Omega^c}^2 d\m \right|= \int \left|2 (\de_{\Omega^c}- \de_{\Omega_i^c}) \de_{\Omega^c} + (\de_{\Omega^c}-\de_{\Omega_i^c})^2\right| d\pi_i \leq  2 \epsilon + \epsilon^2
\end{align*}if $i\geq j_{\delta}$
for $j_\delta\in \mathbb N$ sufficiently large. 
 Hence $\de_{\Omega_i^c}$ converges $L^2$-strongly to $\de_{\Omega^c}$.

Let $\varphi^k\in C_b(\R)$ be sequence of continuous functions such that $\varphi^k\uparrow 1_{[\eta,\infty)}$ pointwise for $\eta>0$.  One can  check that $h^k_i=\varphi^k\circ \de_{\Omega_i^c}\in C_b(X)$ converges uniformily to $h^k=\varphi^k\circ \de_{\Omega^c}$, and in particular there exists $i_\epsilon\in \mathbb N$ such that 
\begin{align*}
\int h^k_i d\m_i \leq  \int h^k d\m+\epsilon \leq \int 1_{[\eta,\infty)}\circ \de_{\Omega^c} d\m+ \epsilon = \m(\de_{\Omega^c}^{-1}([\eta, \infty)))+\epsilon
\end{align*}
for $i\geq i_\epsilon.$
For $k\rightarrow \infty$ we obtain $h^k_i \rightarrow 1_{[\eta, \infty)}\circ \de_{\Omega_i^c}=1_{\de_{\Omega_i^c}^{-1}([\eta,\infty))}$ and
\begin{align*}
\m_i(\de_{\Omega_i^c}^{-1}([\eta, \infty))) \leq \m(\de_{\Omega^c}^{-1}([\eta, \infty))) +\epsilon.
\end{align*}
Finally, we take $\eta\downarrow 0$, $i\rightarrow \infty$ and $\epsilon\downarrow 0$ in this order. It follows 
\begin{align*}
\limsup_{i\rightarrow \infty} \m_i(\Omega_i)\leq \m(\Omega).
\end{align*}
Corollary \ref{cor:constgrad} implies
\begin{align*}
\limsup_{i\rightarrow \infty}\int |\nabla \de_{\Omega^c_i}|^2 d\m_{i} =\limsup_{i\rightarrow \infty}\m_i(\Omega_i) \leq \m(\Omega)= \int |\nabla \de_{\Omega^c}|^2 d\m .
\end{align*}
Hence $\de_{\Omega_i^c}$ converges $H^{1,2}$-strongly to $\de_{\Omega^c}$.

Let $x\in \Omega$ be arbitrary. Then, there exists $\delta>0$ such that $B_\delta(x)\subset \Omega$ and there exists a sequence $x_i\in \Omega_i$ such that $x_i\rightarrow x$, $B_\delta(x_i)\subset \Omega_i$ and $\overline{B_\delta(x_i)}$ converges in Gromov-Hausdorff sense ot $\overline{B_\delta(x)}$. 

We recall the following lemma {\cite[Lemma 2.10]{ambrosio_honda}}.
\begin{lemma}
For any $\phi \in \lip(X)$ with $\supp \phi\subset B_\delta(x)$ there exists  a sequence $\phi_i\in \lip(X_i)$ with $\supp \phi_i\subset B_\delta(x_i)$ such that $\sup_{i} \lip \phi_i<\infty$ and $\phi_i$ converges $H^{1,2}$-strongly to $\phi$.
\end{lemma}
Hence,  given $\phi$ and $\phi_i$ as in the previous lemma $H^{1,2}$-strong convergence of $\de_{\Omega_i^c}$ to $\de_{\Omega^c}$ together with \eqref{rcdinnerproduct} yields
\begin{align*}
\int \langle \nabla \de_{S_i}, \nabla \phi_i\rangle d\m_{X_i} \rightarrow \int \langle \nabla \de_S, \nabla \phi\rangle d\m_{X}.
\end{align*}
Set $f_{K,N,H}=  {s'_{\frac{K}{N-1}, \frac{H}{N-1}}}/{s_{\frac{K}{N-1}, \frac{H}{N-1}}} $. Since $H_i\rightarrow H$, it follows
\begin{align*}
f_{K, N, H_i} \rightarrow f_{K,N,H}
% \frac{s'_{\frac{K}{N-1}, \frac{H_i}{N-1}}}{s_{\frac{K}{N-1}, \frac{H_i}{N-1}}} \rightarrow  \frac{s'_{\frac{K}{N-1}, \frac{H}{N-1}}}{s_{\frac{K}{N-1}, \frac{H}{N-1}}}
 \ \mbox{locally uniformily.}
\end{align*}
Hence, the composition $f_{K,N,H_i}\circ \de_{\Omega_i^c}$ converges uniformly to $f_{K,N,H}\circ \de_{\Omega^c}$,
%$
% ({s'_{\frac{K}{N-1}, \frac{H_i}{N-1}}}/{s_{\frac{K}{N-1}, \frac{H_i}{N-1}}})\circ \de_{\Omega_i^c}$ converges uniformly to $ ({s'_{\frac{K}{N-1}, \frac{H}{N-1}}}/{s_{\frac{K}{N-1}, \frac{H}{N-1}}})\circ \de_{\Omega^c}$, 
and hence $L^2$-strongly. Therefore
\begin{align*}
\int \phi_i  \frac{s'_{\frac{K}{N-1}, \frac{H_i}{N-1}}(u_i)}{s_{\frac{K}{N-1}, \frac{H_i}{N-1}}(u_i)}d\m_{X_i}\rightarrow 
\int \phi  \frac{s'_{\frac{K}{N-1}, \frac{H}{N-1}}(u)}{s_{\frac{K}{N-1}, \frac{H}{N-1}}(u)}d\m_X
\end{align*}
By locality of the distributional Laplacian, this implies the desired estimate.
\end{proof}{
\begin{jjj}
As the referee pointed out to the author that  a similar strategy as in the previous proof  is applied in \cite{aht_weyl} where it is proved that for sequences of uniformly continuous functions, $L^2$-convergence
and uniform convergence are equivalent.
\end{jjj}}
{Theorem \ref{th:meancurvaturestability}, Lemma \ref{lem:AA},  compactness of $\RCD$ spaces w.r.t. $\mathbb D$, the Arzela-Ascoli theorem and the definition of the uniform distance $\mathcal D$ imply the following compactness theorem.
\begin{corollary}\label{cor:com}
%Let $K\in \mathbb R$ and $N\in (1,\infty)$. 
%Let  $(X_i, \de_i, \m_i)_{i\in \mathbb N}$, be a sequence of $\RCD(K,N)$ spaces that converges in measured Gromov-Hausdorff sense to a compact $\RCD(K,N)$ space $(X, \de, \m)$, and let $\Omega_i\subset X_i$ be open sets such that $\partial \Omega_i$ has Laplace mean curvature bounded from below by $H$.  Then, there exists a subsequence $(i_j)_{j\in \mathbb N}$  and an open set $\Omega\subset X$ with Laplace mean curvature bounded from below by $H$ such that $\Omega_{i_j}$ converges w.r.t. $\mathcal D$ to $\Omega$.
Given $K, H\in \R, N\in [1,\infty)$ and $D>0$ the family $\mathcal M(K,N, D, H)$ of pairs $(X,\Omega)$ for  a compact, normalized $\RCD(K,N)$ space $X$ with $\diam_X\leq D$ and an open subset $\Omega\subset X$  with $\partial \Omega$ having Laplace mean curvature bounded from below by $H$ is compact w.r.t. $\mathbb D + \mathcal D$ where $(\mathbb D + \mathcal D)((X,\Omega), (\tilde X, \tilde \Omega))= \mathbb D(X,\tilde X)+ \mathcal D(\Omega, \tilde \Omega)$.
\end{corollary}}

\begin{theorem} Let $\Gamma, D, c, C>0$,  $N>1$ and $m\in\mathbb N\backslash \{1\}$. For every $\epsilon >0$ there exists $\delta>0$ such that the following holds. 

Let  $X$ be a normalized $\RCD(-\delta, N)$ space with $\diam_X\leq D$ and let $\Omega_\alpha\subset X$ be open subsets $\Omega_\alpha\subset X$, $\alpha=1, \dots, m$, such that $\Omega_\alpha$ is $(c,C)$-uniform, $\Omega_\alpha$ has Laplace mean curvature bounded from $-\delta$ and $\de(\Omega_\alpha^c, \Omega_\beta^c)\geq \Gamma>0$ for  $\alpha\neq \beta$. Set $\Omega=\bigcap_{\alpha=1}^m \Omega_\alpha$.

Then, $m=2$ and there exist $D>0$, an $\RCD(0, N)$ space $Z$, an $\RCD(0,N-1)$ space $Y$ and  an open subset $\Omega'\subset Z$ such that
$(\tilde{\Omega}', \tilde \de_{\Omega'}, \m_Z|_{\Omega'}) \simeq Y\otimes [0, D]$ and 
\begin{align}\label{distances}
\mathbb D(X, Z)\leq \epsilon \ \ \ \mbox{ and } \ \ \ \ \mathcal D({\Omega},  {\Omega'})\leq \epsilon.
\end{align}
%
% $(\Omega, \tilde\de_{\Omega}, \m_X|_{\Omega})$ is isomorphic to $Y\otimes [0, D]$ such that $X_i$ subconverges in measured Gromov-Hausdorff sense to $X$ and $\de_{S_{i}}^{\alpha}$ subconverges in Arzela-Ascoli sense to $\de_{S^{\alpha}}$.
\end{theorem}
\begin{proof}
We assume, there exists a sequence of $\RCD(-\frac{1}{i}, N)$ spaces $X_i$ with subsets $\Omega_{\alpha, i}$ that satisfy the assumptions in the theorem but fail the second claim in \eqref{distances} for $\epsilon>0$. 

By stability and compactness of the class of $\RCD$ spaces w.r.t. measured GH convergence there exists an $\RCD(0,N)$ space $Z$ such that a subsequence of $X_i$, that by abuse of notation we also call $X_i$, converges in measured Gromov-Hausdorff sense to $Z$.  Hence, there exists $i_\epsilon\in \mathbb N$ such that $\mathbb D(X,Z)<\epsilon$ for $i\geq i_\epsilon$. After extracting another subsequence $\de_{\Omega^c_{i, \alpha}}$ converges uniformly to $\de_{\Omega^c_{\alpha}}$ for open subsets $\Omega_\alpha\subset Z$, $\alpha=1, \dots, k$ where $k\leq m$. By Theorem \ref{th:meancurvaturestability} $\de_{\Omega_{\alpha}^c}|_{\Omega_\alpha} =: u^\alpha$ satisfies
\begin{align*}
{\bf \Delta}_{\Omega^{}} u^\alpha\geq 0
\end{align*}
i.e. $\Omega_\alpha$ has Laplace mean curvature bounded from below.

By Lemma \ref{lem:uniformdomain} $\Omega_\alpha$ is a $(c,C)$-uniform domain and in particular connected. Hence $\Omega=\bigcap_\alpha \Omega_\alpha$ is connected. Moreover $\de(\Omega_\alpha^c, \Omega_\beta^c)\geq \Gamma$ for all $\alpha\neq \beta$. As in Lemma \ref{lem:harmonic} we derive that $k=2$ and that $u^\alpha$ is harmonic on $\Omega$. Hence, $(\tilde \Omega, \tilde \de_{\Omega}, \m_\Omega)$ is isomorphic to $Y\otimes [0,D]$ for an $\RCD(0,N-1)$ space $Y$. 

On the other hand, uniform convergence of $\de_{\Omega_{\alpha,i}^c}$ to $\de_{\Omega_\alpha^c}$ for all $\alpha=1, \dots, m$ implies $m=2$ and
\begin{align*} 
\mathcal D(\Omega_i, \Omega)\leq \epsilon
\end{align*}
for $i$ sufficiently large by definition of $\mathcal D$. This is a contradiction.
\end{proof}

Very similarly one can prove the following result which is an almost rigidity statement that corresponds to the main rigidity theorem in \cite{bkmw}.

\begin{theorem} Let $D, c, C>0$ and $N>1$. For every $\epsilon >0$ there exists $\delta>0$ such that the following holds. 

Let $X$ be a normalized $\RCD(-\delta, N)$ space with $\diam_X\leq D$ and let $\Omega$ be  open and $(c,C)$-uniform with  Laplace mean curvature bounded from below by $N-1-\delta$.  Assume there exists $x\in \Omega$ such that $\de_{\Omega^c}(x)\geq 1-\delta$.

Then, there exists an $\RCD(0, N)$ space $Z$, an $\RCD(N-2,N-1)$ space $Y$ and  an open subset $\Omega'\subset Z$ such that
$(\tilde{\Omega}', \tilde \de_{\Omega'}, \m_Z|_{\Omega'})$ is isomorphic to $Y\times^{N-1}_{r} [0, 1]$ and 
\begin{align*}
\mathbb D(X,Z)\leq \epsilon \ \ \ \mbox{ and } \ \ \ \ \mathcal D(\Omega, \Omega')\leq \epsilon.
\end{align*}
%
% $(\Omega, \tilde\de_{\Omega}, \m_X|_{\Omega})$ is isomorphic to $Y\otimes [0, D]$ such that $X_i$ subconverges in measured Gromov-Hausdorff sense to $X$ and $\de_{S_{i}}^{\alpha}$ subconverges in Arzela-Ascoli sense to $\de_{S^{\alpha}}$.
\end{theorem}
\appendix
\section{Stability of constant mean curvature sets}\label{53}
%We can prove a theorem in the context of $\RCD$ spaces. 

The definition of  Laplace mean curvature bounds motivates us to  say that the boundary $\partial \Omega$ of an open subset $\Omega$ in a compact $\RCD$ space $X$ is a  {\it generalized CMC hypersurface with curvature $H\in \R$ (a generalized minimal hypersurface if $K=0$)} if $\m(\partial \Omega)=0$ and the signed distance function $\de_{\partial \Omega}:= \de_{\overline \Omega}- \de_{\Omega^c}$ satisfies 
\begin{align}\label{ineq:min}
{\bf \Delta}_{\Omega} (\de_{\partial \Omega})\geq - (N-1) \frac{s'_{\frac{K}{N-1}, \frac{H}{N-1}}(-\de_{\partial \Omega})}{s_{\frac{K}{N-1}, \frac{H}{N-1}}(-\de_{\partial \Omega})}\m|_{ \Omega} \ \mbox{ on } \Omega
\end{align}
and 
\begin{align}\label{ineq:min2}
{\bf \Delta}_{X\backslash \overline\Omega} (-\de_{\partial \Omega})\geq - (N-1) \frac{s'_{\frac{K}{N-1}, \frac{-H}{N-1}}(\de_{\partial \Omega})}{s_{\frac{K}{N-1}, \frac{-H}{N-1}}(\de_{\partial \Omega})}\m|_{X\backslash \overline \Omega} \ \mbox{ on }X\backslash \overline \Omega.
\end{align}
By symmetry in $\Omega$ and $(\Omega^c)^\circ$, $\partial \Omega$ has constant mean curvature $H$ if and only if $\partial \Omega^c$ has constant mean curvature $-H$.

When $\Omega$ is  a subset with smooth boundary in a Riemannian manifold with  Ricci curvature bounded from below by  $K$ \eqref{ineq:min} and \eqref{ineq:min2} are equivalent to $\partial \Omega$ being a CMC hypersurface, as recently discussed  in \cite{ms21} for $K=0$.  In nonsmooth setting one can find examples that satisfy these estimates for every $H\in [-1,1]$ (Example \ref{ex:ex}). Therefore it is suggested by the authors in \cite{apps} to say the boundary of $\Omega$ has {\it a mean curvature barier $H$.} We will adapt this in the following.

For stability of this notion we encounter the following problem: The uniform limit of a signed distance functions $\de_{\partial \Omega_i}$ on $\RCD(K,N)$ spaces $X_i$ may not be a signed distance function of a set $\Omega$ with $\m(\partial \Omega)=0$.
But assuming a uniform inner/outer ball condition for $\Omega$ {(Definition \ref{def:innerouter})} one can prove the following lemma. 

\begin{lemma}\label{lemma:a} Let $K\in \mathbb R$, $N\in (1,\infty)$ and $\delta>0$.
Let  $(X_i, \de_i, \m_i)_{i\in \mathbb N}$, be a sequence of $R\CD(K,N)$ spaces that converges in measured Gromov-Hausdorff sense to a compact metric measure space $(X, \de, \m)$, and let $\Omega_i\subset X_i$ be open sets with $\m_i(\partial S)=0$ that satisfy a $\delta$-uniform outer/inner ball condition. Set $\partial \Omega_i=S_i$.  Then, $\de_{S_i}: X_i\rightarrow \mathbb R$ subconverges in Arzela-Ascoli sense to a $1$-Lipschitz function $u:X\rightarrow \R$ that is the signed distance function of  $\partial \Omega$ with $\Omega= u^{-1}((-\infty,0))$.
\end{lemma} 
\begin{definition}[Outer and inner ball condition]\label{def:innerouter}
Let $(X,\de)$ be a metric space.
Let $\Omega\subset X$ and $\partial \Omega =S$. We say $S$ satisfies an outer ball condition in  $x\in S$ if there exists $r_{\red x}>0$ and $p_x\in \Omega^c$ such that $\de(x,p_x)=r_{\red x}$ and $B_{r_{\red x}}(p_x)\subset \Omega^c$. 
{We say $S$ satisfies an outer ball condition if it satisfies an exterior ball condition in every $x\in S$. Moreover $S$ satisfies a uniform $\delta$-outer ball condition if there exists $\delta>0$ such that $r_x\geq \delta$ for all $x\in S$.}

Similar, $\Omega$ satisfies an inner (uniform $\delta$-inner) ball condition if the previous definition holds with $\Omega^c$ replaces with $\Omega$.
\end{definition}
\begin{proof}[Proof of Lemma \ref{lemma:a}] Form the Lemma \ref{lem:AA} we conclude that $\de_{S_i}$ subconverges uniformly to a function $u$ such that $u= - \de_{\Omega_1^c}$ on $\Omega_1$ and $u= \de_{\Omega_2}^c$ on ${\Omega_2}$ where $\Omega_1=u^{-1}((-\infty, 0))$ and $\Omega_2=u^{-1}((0,\infty))$. 

We only have to show $\partial \Omega_1=\partial \Omega_2= u^{-1}(\{0\})$. By symmetry we only have to prove the first equality. For that we set $\Omega_1=\Omega$.
We know that $\partial \Omega \subset u^{-1}(0)$. Pick $x\in u^{-1}(0)$. Then, there exist $x_i$ with $\de_{S^i}(x_i)=0$ such that $x_i\rightarrow x$. Since $\Omega_i$ satisfies a $\delta$-uniform outer/inner ball condition there exist geodesics $\gamma_i: [-\delta,\delta] \rightarrow X_i$ with $\gamma_i(0)=x_i$, $\gamma_i([-\delta, 0)) \subset \Omega$ and $\gamma_i((0,\delta])\subset \Omega^c_i$. Moreover $\gamma_i$ converges uniformily to geodesic $\gamma: [-\delta,\delta]\rightarrow X$ with $\gamma([-\delta, 0))\subset u^{-1}((-\infty, 0))$ and $\gamma((0,\delta])\subset u^{-1}((0,\infty))$. Hence $x\in \partial \Omega$.
\end{proof}
\begin{theorem}\label{th:min}
Let $K\in \R$, $D, \eta>0$ and $N\in [2,\infty)$. For $\epsilon>0$ there exists $\delta>0$ such that the following holds. 

Let $X_i$ be a sequence of $\RCD(K,N)$ spaces with $\diam_X\leq D$ and let $\Omega_i\subset X_i$ be open subsets that satisfy a $\eta$-uniform inner-outer ball condition and such that $\partial \Omega_i$ have a mean curvature  barrier $H\in \R$ in the sense of \eqref{ineq:min} and \eqref{ineq:min2}.

Then, there exists a measured GH converging subsequence of $X_i$ with a limit $\RCD(K,N)$ space $X$ such that a subsequence of $\de_{\partial \Omega_i}$ uniformly converges to $\de_{\partial \Omega}$ for an open subset $\Omega$ in $X$ that has a mean curvature barrier $H$.
\end{theorem}
\begin{proof}[Proof of Theorem \ref{th:min}]
The Theorem follows now from stability of mean curvature bounds together with the previous lemma.
\end{proof}
\begin{remark} \label{rem:effective}{
In general CMC hypersurfaces don't satisfy an effective $\delta$-uniform outer/inner ball condition with $\delta$ only depending on geometric information of $X$ and the mean curvature $H$. Counter-examples are families of catenoids in $\mathbb R^3$ with increasingly big second fundamental form.  On the other hand a regularity theory for perimeter minimizing sets and for isoperimetric sets in the context of $\RCD$ spaces was developped in recent work by Mondino and Semola \cite{ms21},  and Antonelli, Pasqualetto, Pozzetta and Semola \cite{apps}. %However their work fully exploits the variational characterization of these sets.
 }
\end{remark}
\begin{example}\label{ex:ex} In the following we give two  examples:
\label{examp}
(1) 
The first example was suggested to the author by Daniele Semola. One can consider the metric (measure) space that is the result of gluing together two copies of $\overline{B_1(0)}\subset \R^2$  along their boundaries. This doubling space $X$ has Alexandrov curvature bounded from below by $0$ and is therefore an $\RCD(0,2)$ space by theorems of Perelman-Petrunin \cite{p, pg, palvs} and Lytchak-Stadler \cite{lytchakstadler}. There is an isometric copy of ${B_1(0)}=\Omega$ inside of $X$ such that $\Omega^c= \overline{B_1(0)}$ and $\partial \Omega \simeq \partial \Omega^c\simeq \partial B_1(0)=: S$. Then  $S$ has Laplace mean curvature bounded from below by $1$, seen both as boundary of $\Omega$ and  as boundary of $\Omega^c$.  Hence $S$ has a mean curvature barrier $H$ for every $H\in [-1, 1]$ in the sense that \eqref{ineq:min} and \eqref{ineq:min2} hold for every $H\in [-1, 1]$.  In particular, it is a generalized minimal surface because one can choose $H=0$.
{The space $X$ can be obtained  as a limit of smooth Riemannian manifolds $M_i$, and the distance function $\de_{\partial \Omega}$ as the limit of distance functions on $M_i$ corresponding to smooth domains $\Omega_i\subset M_i$. More precisely, as consequence of the proof of the double space theorem in smooth context  $X$ can be constructed as the $C^0$-limit of Riemannian spheres with curvature bounded from below by $0$, and $\Omega$ is the $\mathcal D$-limit of balls with constant mean curvature $H$ for a given  $H\in [-1,1]$. }

(2)
{Another example   the referee suggested is the double space  $X$ of two copies $D_1$ and $D_2$ of a convex domain $D$ with smooth boundary  in $\R^2$ such that the second fundamental form of $\partial D$ is non-negative and not necessarily positive.  Again $X$ is an $\RCD(0,2)$ space. In this situation $\partial D_1\simeq \partial D_2=S$ is Laplace mean convex  as the boundary of  $D_1$ but also  as the boundary of $D_2$. Hence, $S\subset X$ has a mean curvature barrier $0$.  Again one may obtain $X$ and $S$ as the limit of smooth Riemannian manifolds with Ricci curvature bounded from below and as the $\mathcal D$-limit of smooth minimal hypersurfaces, respectively. 
The same construction also works in higher dimensions. }

{The hypersurfaces presented in (1) and (2) are  not minimal hypersurfaces in the classical sense or locally perimeter minimizing in the sense of \cite{ms21}.
%This may indicate that one should strengthen the definitions of minimal and CMC hypersurface to obtain a class of objects in $\RCD$ spaces  that resemble better their smooth counterparts. 
%On the other hand the boundary of perimeter minimizing as well as isoperimetric sets in $\RCD$ spaces  have a mean curvature barriere,  according to  \cite{ms21} and \cite{apps}.
%Though the two  examples are not isoperimetric nor locally perimeter minimizing,
But they are  "equatorial" inside of the ambient space and may  emerge as the solution of a variational problem,  for instance a min-max problem, like the equator in a sphere of constant curvature.
%The low regularity of the ambient space prevents that they  satisfy  classical equations. But  they are minimal hypersurfaces in the proposed sense.
}
\end{example}
\small{
\bibliography{new}

\providecommand{\bysame}{\leavevmode\hbox to3em{\hrulefill}\thinspace}
\providecommand{\MR}{\relax\ifhmode\unskip\space\fi MR }
\providecommand{\MRhref}[2]{%
  \href{http://www.ams.org/mathscinet-getitem?mr=#1}{#2}
}
\providecommand{\href}[2]{#2}
\begin{thebibliography}{BKMW20}

\bibitem[AGS14]{agsriemannian}
Luigi Ambrosio, Nicola Gigli, and Giuseppe Savar{\'e}, \emph{Metric measure
  spaces with {R}iemannian {R}icci curvature bounded from below}, Duke Math. J.
  \textbf{163} (2014), no.~7, 1405--1490. \MR{3205729}

\bibitem[AGS15]{agsbakryemery}
Luigi Ambrosio, Nicola Gigli, and Giuseppe Savar\'e, \emph{{Bakry-\'Emery}
  curvature-dimension condition and {R}iemannian {R}icci curvature bounds},
  Ann. Probab. \textbf{43} (2015), no.~1, 339--404. \MR{3298475}

\bibitem[AH18]{ambrosio_honda}
Luigi Ambrosio and Shouhei Honda, \emph{Local spectral convergence in
  {$RCD^*(K,N)$} spaces}, Nonlinear Analysis, Theory, Methods and Applications
  \textbf{177} (2018), 1--23 (English).

\bibitem[AHT18]{aht_weyl}
Luigi Ambrosio, Shouhei Honda, and David Tewodrose, \emph{Short-time behavior
  of the heat kernel and {Weyl}'s law on {{\(\mathrm{RCD}^*(K,N)\)}} spaces},
  Ann. Global Anal. Geom. \textbf{53} (2018), no.~1, 97--119 (English).

\bibitem[AMS19]{amsnonlinear}
Luigi Ambrosio, Andrea Mondino, and Giuseppe Savar\'{e}, \emph{Nonlinear
  {D}iffusion {E}quations and {C}urvature {C}onditions in {M}etric {M}easure
  {S}paces}, Mem. Amer. Math. Soc. \textbf{262} (2019), no.~1270, 0.
  \MR{4044464}

\bibitem[APPS22]{apps}
Gioacchino {Antonelli}, Enrico {Pasqualetto}, Marco {Pozzetta}, and Daniele
  {Semola}, \emph{{Sharp isoperimetric comparison and asymptotic isoperimetry
  on non collapsed spaces with lower Ricci bounds}}, arXiv e-prints (2022),
  arXiv:2201.04916.

\bibitem[BB11]{bjoern}
Anders Bj{\"o}rn and Jana Bj{\"o}rn, \emph{Nonlinear potential theory on metric
  spaces}, EMS Tracts in Mathematics, vol.~17, European Mathematical Society
  (EMS), Z\"urich, 2011. \MR{2867756}

\bibitem[BKMW20]{bkmw}
Annegret Burtscher, Christian Ketterer, Robert~J. McCann, and Eric Woolgar,
  \emph{Inscribed radius bounds for lower {R}icci bounded metric measure spaces
  with mean convex boundary}, SIGMA Symmetry Integrability Geom. Methods Appl.
  \textbf{16} (2020), Paper No. 131, 29. \MR{4185085}

\bibitem[BNS22]{bns}
Elia Bru{\`e}, Aaron Naber, and Daniele Semola, \emph{Boundary regularity and
  stability for spaces with {Ricci} bounded below}, Invent. Math. \textbf{228}
  (2022), no.~2, 777--891 (English).

\bibitem[BS10]{bast}
Kathrin Bacher and Karl-Theodor Sturm, \emph{Localization and tensorization
  properties of the curvature-dimension condition for metric measure spaces},
  J. Funct. Anal. \textbf{259} (2010), no.~1, 28--56. \MR{2610378
  (2011i:53050)}

\bibitem[Cav14]{cavom}
Fabio Cavalletti, \emph{Monge problem in metric measure spaces with
  {R}iemannian curvature-dimension condition}, Nonlinear Anal. \textbf{99}
  (2014), 136--151. \MR{3160530}

\bibitem[Che99]{cheegerlipschitz}
Jeff Cheeger, \emph{Differentiability of {L}ipschitz functions on metric
  measure spaces}, Geom. Funct. Anal. \textbf{9} (1999), no.~3, 428--517.
  \MR{1708448 (2000g:53043)}

\bibitem[CK92]{crokekleiner}
Christopher~B. Croke and Bruce Kleiner, \emph{A warped product splitting
  theorem}, Duke Math. J. \textbf{67} (1992), no.~3, 571--574. \MR{1181314}

\bibitem[CM17]{cavmon}
Fabio Cavalletti and Andrea Mondino, \emph{Sharp and rigid isoperimetric
  inequalities in metric-measure spaces with lower {R}icci curvature bounds},
  Invent. Math. \textbf{208} (2017), no.~3, 803--849. \MR{3648975}

\bibitem[CM20a]{cav-mon-lapl-18}
\bysame, \emph{New formulas for the {L}aplacian of distance functions and
  applications}, Anal. PDE \textbf{13} (2020), no.~7, 2091--2147. \MR{4175820}

\bibitem[CM20b]{CavallettiMondino20+}
Fabio {Cavalletti} and Andrea {Mondino}, \emph{{Optimal transport in Lorentzian
  synthetic spaces, synthetic timelike Ricci curvature lower bounds and
  applications}}, arXiv e-prints (2020), arXiv:2004.08934.

\bibitem[CM21]{cavmil}
Fabio Cavalletti and Emanuel Milman, \emph{The globalization theorem for the
  curvature-dimension condition}, Invent. Math. \textbf{226} (2021), no.~1,
  1--137. \MR{4309491}

\bibitem[EG99]{EvansGangbo99}
L.~C. Evans and W.~Gangbo, \emph{Differential equations methods for the
  {M}onge-{K}antorovich mass transfer problem}, Mem. Amer. Math. Soc.
  \textbf{137} (1999), no.~653, viii+66. \MR{1464149}

\bibitem[EKS15]{erbarkuwadasturm}
Matthias Erbar, Kazumasa Kuwada, and Karl-Theodor Sturm, \emph{On the
  equivalence of the entropic curvature-dimension condition and {B}ochner's
  inequality on metric measure spaces}, Invent. Math. \textbf{201} (2015),
  no.~3, 993--1071. \MR{3385639}

\bibitem[FM02]{FeldmanMcCann02}
Mikhail Feldman and Robert~J. McCann, \emph{Monge's transport problem on a
  {R}iemannian manifold}, Trans. Amer. Math. Soc. \textbf{354} (2002), no.~4,
  1667--1697. \MR{1873023}

\bibitem[Gig15]{giglistructure}
Nicola Gigli, \emph{On the differential structure of metric measure spaces and
  applications}, Mem. Amer. Math. Soc. \textbf{236} (2015), no.~1113, vi+91.
  \MR{3381131}

\bibitem[Gig18]{giglinonsmooth}
\bysame, \emph{Nonsmooth differential geometry---an approach tailored for
  spaces with {R}icci curvature bounded from below}, Mem. Amer. Math. Soc.
  \textbf{251} (2018), no.~1196, v+161. \MR{3756920}

\bibitem[GL80]{Gromov-Lawson-1980}
Mikhael Gromov and H.~Blaine Lawson, Jr., \emph{Spin and scalar curvature in
  the presence of a fundamental group. {I}}, Ann. of Math. (2) \textbf{111}
  (1980), no.~2, 209--230.

\bibitem[GM13]{giglimondino}
Nicola Gigli and Andrea Mondino, \emph{A {PDE} approach to nonlinear potential
  theory in metric measure spaces}, J. Math. Pures Appl. (9) \textbf{100}
  (2013), no.~4, 505--534. \MR{3102164}

\bibitem[GR19]{gigli_rigoni}
Nicola Gigli and Chiara Rigoni, \emph{A note about the strong maximum principle
  on {RCD} spaces}, Canad. Math. Bull. \textbf{62} (2019), no.~2, 259--266.
  \MR{3952516}

\bibitem[Gro19]{gromovmean}
Misha Gromov, \emph{Mean curvature in the light of scalar curvature}, Ann.
  Inst. Fourier (Grenoble) \textbf{69} (2019), no.~7, 3169--3194. \MR{4286832}

\bibitem[Kas83]{Kasue83}
Atsushi Kasue, \emph{Ricci curvature, geodesics and some geometric properties
  of {R}iemannian manifolds with boundary}, J. Math. Soc. Japan \textbf{35}
  (1983), no.~1, 117--131. \MR{679079}

\bibitem[Ket20]{kettererHK}
Christian Ketterer, \emph{The {H}eintze-{K}archer inequality for metric measure
  spaces}, Proc. Amer. Math. Soc. \textbf{148} (2020), no.~9, 4041--4056.
  \MR{4127847}

\bibitem[KK20]{Kap-Ket-18}
Vitali Kapovitch and Christian Ketterer, \emph{C{D} meets {CAT}}, J. Reine
  Angew. Math. \textbf{766} (2020), 1--44. \MR{4145200}

\bibitem[KKL23]{kkl}
Christian Ketterer, Yu~Kitabeppu, and Sajjad Lakzian, \emph{The rigidity of
  sharp spectral gap in non-negatively curved spaces}, Nonlinear Anal.
  \textbf{228} (2023), Paper No. 113202. \MR{4526556}

\bibitem[KKS20]{kks}
Vitali {Kapovitch}, Christian {Ketterer}, and Karl-Theodor {Sturm}, \emph{{On
  gluing Alexandrov spaces with lower Ricci curvature bounds}}, to appear in
  Com. Anal. Geom. (2020), arXiv:2003.06242.

\bibitem[Kla17]{Klartag17}
Bo'az Klartag, \emph{Needle decompositions in {R}iemannian geometry}, Mem.
  Amer. Math. Soc. \textbf{249} (2017), no.~1180, v + 77. \MR{3709716}

\bibitem[LS22]{lytchakstadler}
Alexander Lytchak and Stephan Stadler, \emph{Ricci curvature in dimension 2},
  J. Eur. Math. Soc. (2022).

\bibitem[LV09]{lottvillani}
John Lott and C{\'e}dric Villani, \emph{Ricci curvature for metric-measure
  spaces via optimal transport}, Ann. of Math. (2) \textbf{169} (2009), no.~3,
  903--991. \MR{2480619 (2010i:53068)}

\bibitem[MS21]{ms21}
Andrea {Mondino} and Daniele {Semola}, \emph{{Weak Laplacian bounds and minimal
  boundaries in non-smooth spaces with Ricci curvature lower bounds}}, arXiv
  e-prints (2021), arXiv:2107.12344.

\bibitem[MW21]{mw}
Kenneth {Moore} and Eric {Woolgar}, \emph{{Bakry-{\'E}mery Ricci Curvature
  Bounds on Manifolds with Boundary}}, arXiv e-prints (2021), arXiv:2110.06524.

\bibitem[Per]{p}
G.~Perelman, \emph{A.{D}. {A}lexandrov's spaces with curvatures bounded from
  below, {II}},
  http://www.math.psu.edu/petrunin/papers/alexandrov/perelmanASWCBFB2+.pdf.

\bibitem[Per16]{peralesheka}
Raquel Perales, \emph{Volumes and limits of manifolds with {R}icci curvature
  and mean curvature bounds}, Differential Geom. Appl. \textbf{48} (2016),
  23--37. \MR{3534433}

\bibitem[Pet97]{pg}
Anton Petrunin, \emph{Applications of quasigeodesics and gradient curves},
  Comparison geometry ({B}erkeley, {CA}, 1993--94), Math. Sci. Res. Inst.
  Publ., vol.~30, Cambridge Univ. Press, Cambridge, 1997, pp.~203--219.
  \MR{1452875}

\bibitem[Pet11]{palvs}
\bysame, \emph{Alexandrov meets {L}ott-{V}illani-{S}turm}, M\"unster J. Math.
  \textbf{4} (2011), 53--64. \MR{2869253 (2012m:53087)}

\bibitem[PW03]{pewifrankel}
Peter Petersen and Frederick Wilhelm, \emph{On {F}rankel's theorem}, Canad.
  Math. Bull. \textbf{46} (2003), no.~1, 130--139. \MR{1955620}

\bibitem[Sak19]{sakurai}
Yohei Sakurai, \emph{Rigidity of manifolds with boundary under a lower
  {B}akry-\'{E}mery {R}icci curvature bound}, Tohoku Math. J. (2) \textbf{71}
  (2019), no.~1, 69--109. \MR{3920791}

\bibitem[Sav14]{savareself}
Giuseppe Savar{\'e}, \emph{Self-improvement of the {B}akry-\'{E}mery condition
  and {W}asserstein contraction of the heat flow in {${\rm RCD}(K,\infty)$}
  metric measure spaces}, Discrete Contin. Dyn. Syst. \textbf{34} (2014),
  no.~4, 1641--1661. \MR{3121635}

\bibitem[Sor17]{sormanisurvey}
Christina Sormani, \emph{Scalar curvature and intrinsic flat convergence},
  Measure theory in non-smooth spaces, Partial Differ. Equ. Meas. Theory, De
  Gruyter Open, Warsaw, 2017, pp.~288--338. \MR{3701743}

\bibitem[Sor18]{saa}
\bysame, \emph{Intrinsic flat {A}rzela-{A}scoli theorems}, Comm. Anal. Geom.
  \textbf{26} (2018), no.~6, 1317--1373. \MR{3936492}

\bibitem[Stu06a]{stugeo1}
Karl-Theodor Sturm, \emph{On the geometry of metric measure spaces. {I}}, Acta
  Math. \textbf{196} (2006), no.~1, 65--131. \MR{2237206 (2007k:53051a)}

\bibitem[Stu06b]{stugeo2}
\bysame, \emph{On the geometry of metric measure spaces. {II}}, Acta Math.
  \textbf{196} (2006), no.~1, 133--177. \MR{2237207 (2007k:53051b)}

\bibitem[Stu18]{sturmconformal}
\bysame, \emph{Ricci tensor for diffusion operators and curvature-dimension
  inequalities under conformal transformations and time changes}, J. Funct.
  Anal. \textbf{275} (2018), no.~4, 793--829. \MR{3807777}

\bibitem[SY79a]{schoenyaustructure}
R.~Schoen and S.~T. Yau, \emph{On the structure of manifolds with positive
  scalar curvature}, Manuscripta Math. \textbf{28} (1979), no.~1-3, 159--183.
  \MR{535700}

\bibitem[SY79b]{schoenyau79}
R.~Schoen and Shing~Tung Yau, \emph{Existence of incompressible minimal
  surfaces and the topology of three-dimensional manifolds with nonnegative
  scalar curvature}, Ann. of Math. (2) \textbf{110} (1979), no.~1, 127--142.
  \MR{541332}

\bibitem[Vil09]{viltot}
C{\'e}dric Villani, \emph{Optimal transport, old and new}, Grundlehren der
  Mathematischen Wissenschaften [Fundamental Principles of Mathematical
  Sciences], vol. 338, Springer-Verlag, Berlin, 2009. \MR{2459454
  (2010f:49001)}

\bibitem[Won08]{wong}
Jeremy Wong, \emph{An extension procedure for manifolds with boundary}, Pacific
  J. Math. \textbf{235} (2008), no.~1, 173--199. \MR{2379775}

\end{thebibliography}

\bibliographystyle{amsalpha}}
\end{document}